\documentclass[reqno]{amsart}
\usepackage{amsmath, amsthm, amssymb, amstext}
\usepackage[foot]{amsaddr}
\usepackage[left=2cm,right=2cm,top=1.5cm,bottom=1.5cm,includeheadfoot]{geometry}
\usepackage{hyperref,xcolor}
\hypersetup{
	pdfborder={0 0 0},
	colorlinks,
}
\usepackage{todonotes}
\usepackage{enumitem}
\newtheorem{theorem}{Theorem}
\newtheorem{remark}[theorem]{Remark}
\newtheorem{lemma}[theorem]{Lemma}
\newtheorem{proposition}[theorem]{Proposition}

\newtheorem{definition}[theorem]{Definition}

\newcommand{\la} {\lambda}

\newcommand{\Om}{\Omega}

\newcommand{\ds} {\displaystyle}
\newcommand{\e}{\epsilon}

\newcommand{\al} {\alpha}
\newcommand{\ba} {\beta}

\newcommand{\ga} {\gamma}

\newcommand{\ra} {\rightarrow}

\newcommand{\De} {\Delta}
\newcommand{\La} {\Lambda}
\newcommand{\noi} {\noindent}

\newcommand{\mb} {\mathbb}
\newcommand{\mc} {\mathcal}

\numberwithin{theorem}{section} \numberwithin{equation}{section}

\title{$p$-biharmonic Kirchhoff equations with critical  Choquard nonlinearity}
\author{divya goel, sarika goyal, diksha saini}
\address{Divya Goel \newline
	Department of Mathematical Sciences, Indian Institute of Technology (BHU), Varanasi, 221005, India.}  
\email{divya.mat@iitbhu.ac.in }

\address{Sarika Goyal \newline
	Department of Mathematics, Netaji Subhash University of Technology, Dwarka sector-3, India}
\email{sarika@nsut.ac.in, sarika1.iitd@gmail.com }
\address{Diksha Saini \newline
	Department of Mathematics,  Netaji Subhash University of Technology, Dwarka sector-3, India}
\email{diksha.saini.phd23@nsut.ac.in }
\begin{document}
 \begin{abstract}
 \noi In this article, we deal with the following involving $p$-biharmonic  critical Choquard-Kirchhoff equation 
		$$ 
		\begin{array}{lr}
			\left(a+b\left(\int_{\mathbb R^N}|\Delta u|^p dx\right)^{\theta-1}\right) \Delta_{p}^{2}u = \alpha \left(|x|^{-\mu}*u^{p^*_\mu}\right)|u|^{p^*_\mu-2}u+ \la f(x) |u|^{r-2} u \; \text{in}\; \mathbb R^N,
		\end{array}
		$$
		where $a\geq 0$, $b> 0$, $0<\mu<N$, $N>2p$, $p\geq 2$, $\theta\geq1$, $\al$ and $\la$ are positive real parameters, $p_{\mu}^{*}= \frac{p(2N-\mu)}{2(N-2p)}$ is the upper critical exponent in the sense of Hardy-Littlewood-Sobolev inequality. The function $f \in L^{t}(\mathbb R^N)$ with $t= \frac{p^{*}}{(p^* -r)}$ if $p<r<p^*:=\frac{Np}{N-2p}$ and $t=\infty$ if $r\geq p^{*}$. We first prove the concentration compactness principle for the $p$-biharmonic Choquard-type equation. Then using the variational method together with the concentration-compactness, we established the existence and multiplicity of solutions to the above problem with respect to parameters $\la$ and \(\alpha\) for different values of $r$. The results obtained here are new even for $p-$Laplacian.
\end{abstract}

\maketitle
\section{Introduction}
This paper is concerned with the existence of solutions of the following Kirchhoff $p$-biharmonic equation involving Choquard nonlinearity
	$$ (\mc P_{\al, \la}) \quad \left\{
	\begin{array}{lr}
	 (a+b\|u\|^{p(\theta -1)})\Delta_{p}^2 u	= \alpha \left(|x|^{-\mu}*u^{p^*_\mu}\right)|u|^{p^*_\mu-2}u + \la f(x) |u|^{r-2} u \; \text{in}\; \mathbb R^N,
	\end{array}
	\right.
	$$
	where  \(N> 2p\), $p\geq 2$, \(a\geq 0\), \(b> 0\), $\theta\geq 1$,
	$\|u\|= \left(\int_{\mathbb R^N}|\Delta u|^p dx\right)^{1/p}$,
	 \(0<\mu<N \),  $f \in L^{t}(\mathbb R^N)$ with $t= \frac{p^{*}}{(p^* - r)}$ if $p<r<p^*:=\frac{Np}{N-2p}$ and $t=\infty$ if $r \geq p^{*}$. Here $p_{\mu}^{*}= \frac{p(2N-\mu)}{2(N-2p)}$ is the upper critical exponent in the sense of Hardy-Littlewood-Sobolev inequality.

     In recent years,  mathematicians have been studying non-local problems, the existence and multiplicity of solutions, due to their vast applications. One of the non-local problems  involves Choquard-type nonlinearity. The nonlinear Choquard  equation
    \begin{align}\label{lr4}
-\Delta u + V(x)u = (I_\alpha * F(u))f(u) + g(x, u) \quad \text{in } \mathbb{R}^N,
\end{align}
 where \(\alpha \in (0, N)\), \(f = F' \in C(\mathbb{R}, \mathbb{R})\), \(g \in C(\mathbb{R}^N \times \mathbb{R}, \mathbb{R})\) and $I_{\al}$ is a Riesz potential with
\[
I_\alpha(x) = \frac{B_\alpha}{|x|^{N-\alpha}} \quad \text{where} \ 
  B_\alpha = \frac{\Gamma\left(\frac{N-\alpha}{2}\right)}{\Gamma\left(\frac{\alpha}{2}\right)\pi^{N/2}2^\alpha},
\]
was first studied by S. Pekar \cite{r22} in 1954, for $V(x) =0$, $F(u)=|u|^2$, $\al=2, ~g(x,u)=0$ and $N=3$, to describe the quantum theory of the polaron at rest. Later, Choquard \cite{r46lieb1977existence} used equation \eqref{lr4}  to study an electron trapped in its hole, and Penrose used equation \eqref{lr4} as a model for self-gravitating matter in \cite{48}, respectively.  In the last decade,  Moroz and Schaftingen \cite{moroz2013,r31,r32} studied the Choquard equations and proved the existence of  ground state solutions. They further proved the regularity, positivity and radial symmetry of the solutions. Carvalho,  Silva, and Goulart \cite{r40goulart10choquard} studied  \eqref{lr4} with  convex-concave type of growth  and established the multiplicity of solutions using  the Nehari method with a fine analysis on the
nonlinear Rayleigh quotient. For a more detailed overview of the Choquard-type equations, we refer readers to  \cite{r26,r27,r39moroz2017guide} and references therein.\\
Another type of non-local operator  which researchers are studying currently is the Kirchhoff operator. In \cite{33} Kirchhoff introduced the following equation 
\begin{align*}
      \rho  u_{tt} - \left( \frac{\rho_0}{h}+ \frac{E}{2L}\int_{0}^{L} \left|u_x\right|^2 dx \right)  u_{xx}  =0, 
  \end{align*}
  where $\rho, \rho_0,h,E,L$ are constants, 
  to extend the classical  D'Alembert's wave equation by incorporating the changes in the length of the string produced by transverse vibrations.  
  To know more about Kirchhoff operator, one can see \cite{34,35,36} and references therein. On the other hand, many authors considered the Choquard-Kirchhoff type problems with  Hardy-Littlewood-Sobolev critical nonlinearities. Due to the presence of Kirchhoff operator and Choquard type nonlinearity, these type of problems are called doubly non-local. As an example, we cite \cite{r7} in which  Liang,  Pucci and  Zhang considered the  following  problem 
	\begin{align*}
	- (a+b\|u\|^{2})\Delta u	= \gamma \left(|x|^{-\mu}*u^{2^*_\mu}\right)|u|^{2^*_\mu-2}u + \la f(x) |u|^{r-2} u \ \text{in}\; \mathbb R^N,
\end{align*}
\noi where $\la, \ga >0$ are parameters. They showed the  existence and multiplicity of solutions for $1<q\leq2$. With no offense of providing the full list, for more details, please refer \cite{r50li2021multiple, r48goel2019kirchhoff, r49wang2019multiplicity} and references therein. 
\\\\
\noi 
 Problems involving quasilinear problems have been the focus of intensive research in recent years.  Problems involving $p$-biharmonic operators are of great importance as they appear in many applications, such as in elasticity and plate theory, Quantum mechanics, Astrophysics, Material sciences, see \cite{lazer1990large, mckenna1987nonlinear, mckenna1990travelling}.
    \begin{align}\label{lr1*}
        \Delta_{p}^{2}u = \la |u|^{p-2} u +  |u|^{r-2} u \ \text{in} \ \Omega ; \quad u=\nabla u= 0 \ \text{in} \ \partial\Omega.
    \end{align}
    where $\Omega$ is a bounded domain in $\mathbb R^N$, $p<r\leq p^*=\frac{Np}{N-2p}$. The authors showed that the problem \eqref{lr1*} possesses infinitely many solutions for $q<p^*$, using Fountain's theorem and for the case $q=p^*$, the existence result is obtained using abstract critical point theory based on a pseudo-index related to the cohomological index. While in \cite{42}, Chung and Minh investigated the $p$-biharmonic Kirchhoff type problem for the bounded domain and proved the existence of a non-trivial solution by the Mountain Pass theorem. On the other hand, there are very few articles available on problems involving the $p$-biharmonic operator over $\mathbb{R}^N$. In \cite{r35liu2016infinitely}, Liu and Chen considered the  following $p$-biharmonic elliptic equation
    \begin{align}\label{lr2}
        \Delta_{p}^{2}u-\Delta_p u+V(x)|u|^{p-2} u = \la f_1(x) |u|^{r-2} u+ f_2(x) |u|^{q-2} u; \ x\in \mathbb R^N,
    \end{align}
    where $2<2p<N$, $1<r<p<q<p^*=\frac{Np}{N-2p}$ and  the potential function $V(x)\in C(\mb R^N)$ satisfies $\displaystyle\inf_{x\in \mathbb R^N}V(x)>0$. Here, they established the existence of infinitely many high-energy solutions to equation \eqref{lr2}, using the variational methods. In \cite{r36bae2019existence},  Bae, Kim, Lee, and Park   examined the following $p$-biharmonic Kirchhoff type equation
    \begin{align*}
        \Delta_{p}^{2}u +M\left(\int_{\mathbb R^N}\phi_0(x,\nabla u)dx\right) \text{div}(\phi(x,\nabla u))+V(x)|u|^{p-2} u= \la g(x,u); \ x\in \mathbb R^N,
    \end{align*}
    where the function $\phi(x,v)$ is of type $|v|^{p-2} v$, $\phi(x,v)=\frac{d}{dv}\phi_0(x,v)$, the potential function $V(x)\in  C(\mb R^N) $ and $g:\mb R^N\times \mb R\ra \mb R$ satisfies the Carath\'{e}odory condition. 
 For $1<p<p^*=\frac{Np}{N-2p}$,   using Mountain Pass theorem and Fountain theorem, the authors proved the existence of a non-trivial weak solution and multiplicity of weak solutions. 
For more results on problems involving $p$-biharmonic operators, we refer   \cite{r33bhakta2015entire,44,45,r37alsaedi2021low}, see also the references therein.\\

  \noi  To the best of our knowledge, there is no result that takes into account the problem involving
the Kirchhoff operator, $p-$biharmonic operator, and the convolution term. Precisely, we will study a new class of equations with a physical,
chemical, and biological background.
    In this paper, we consider the critical Choquard-Kirchhoff equation involving $p$-biharmonic operator in the whole space $\mb R^N$. Motivated by the results introduced in \cite{r7, 43} and a few papers on the  $p$-biharmonic operator, we established the existence and multiplicity of solutions to the problem $(\mc P_{\al, \la})$  for the non-degenerate and degenerate case. 
 In this article, we established the existence and multiplicity of solutions depending on a different range of $r$ and $\theta$. 
The main difficulty here is the lack of compactness of the embedding. To solve this issue, we established the concentration-compactness lemma for the biharmonic operator with Choquard nonlinearity. Furthermore, we used variational methods to prove the existence and multiplicity of solutions.

\noi Before stating our main result, we assume the following condition on the function $f:\mathbb R^N \rightarrow\mathbb R$,
\begin{enumerate}
\item[$(f_1)$] $f\geq 0$ with $f\not\equiv 0$ in $\mathbb R^N$, $f\in L^{\frac{p^*}{p^*-r}}$.
\item[$(f_2)$] Let $\Omega:= \{x\in \mb R^N: f(x)>0\}$ be an open subset of $\mathbb R^N$ with $0<|\Omega|<\infty.$
\end{enumerate}

\noi The main results of this article are as follows:
\begin{theorem}\label{th1}
Let $1<r<p$, $p\theta  < 2p_{\mu}^{*}$. Suppose $f$ satisfies $(f_1)$ and $(f_2)$.
Then
\begin{enumerate}
			\item[(i)] For each $\lambda>0$ there exists $\overline{\La} >0$ such that  for all $\al \in (0, \overline{\La})$, the problem $(\mc P_{\al, \la})$ has a sequence of non-trivial solutions $\{u_n\}$ with $\mc E_{\al,\la}\leq 0$ and $u_{n}\rightarrow 0$, as $n\ra \infty$.
			\item[(ii)] For each $\alpha>0$ there exists $\underline{\La} >0$ such that for all $\lambda \in (0, \underline{\La})$, the problem $(\mc P_{\al, \la})$ has a sequence of non-trivial solutions $\{u_n\}$ with $\mc E_{\al,\la}\leq 0$ and $u_{n}\rightarrow 0$, as $n\ra \infty$.
		\end{enumerate}
\end{theorem}
\begin{theorem}\label{th12}
Let $r=p$,  $p\theta  < 2p_{\mu}^{*}$. Suppose $f$ satisfy $(f_1)$.
Then there exists a positive constant $\hat a$ such that for each $a>\hat a$ and $\la\in (0, aS\|f\|^{-1}_{\frac{}{}})$, the problem $(\mc P_{\al,\la})$ has at least $n$ pairs of non-trivial solutions.
\end{theorem}


\begin{theorem}\label{th13}
Let $p\theta \leq r<p^*$,  $p\theta  < 2p_{\mu}^{*}$. Suppose $f$ satisfy $(f_1)$.
Then there exists a positive constant $\tilde{a}$ such that for each $a>\tilde a$ and for all  $\la>0$, the problem $(\mc P_{\al,\la})$ has infinitely many non-trivial solutions.
\end{theorem}


\begin{theorem}\label{th3}
    Let  $p\geq 2$, $p<r<p^*$, $f\geq 0$ with $f\not\equiv 0$ in $\mathbb R^N$, and $p\theta \geq 2p_{\mu}^{*}$ . If either $p\theta = 2p_{\mu}^{*}$, $a>0$ and $b> 2^{p} S_{H,L}^{-\frac{2p_{\mu}^{*}}{p}}$ or $p\theta> 2 p_{\mu}^{*}$, $a>0$ and 
    \begin{align}\label{eb}
    b>\frac{2 p_{\mu}^{*}-p}{p(\theta-1)}\left[\left(\frac{a p(\theta-1)}{(p\theta- 2p_{\mu}^{*})}\right)^{\frac{2p_{\mu}^{*}-p\theta }{2p_{\mu}^{*}-p}} \left(2^p S_{H,L}^{-\frac{2p_{\mu}^{*}}{p}}\right)^{\frac{p(\theta-1)}{2p_{\mu}^{*}-p}}\right]:=b^{1},
    \end{align}
    then there exists $\lambda^*>0$ such that the problem $(\mc P_{\al,\la})$ admits at least two non-trivial solutions in $D^{2,p}(\mathbb R^N)$ for all $\lambda>\lambda^*$.   \end{theorem}

\begin{theorem}\label{th4}
(Degenerate Case)
Let  $1<r<p^*$, $f$ satisfies $(f_1)$ and $p\theta \geq 2p_{\mu}^{*}$. If  $a=0$ and $b> 2^{p} S_{H,L}^{-\frac{2p_{\mu}^{*}}{p}}$, then for all $\la>0$, the problem $\mc(P_{\al,\la})$ admits infinitely many pairs of distinct solutions in $D^{2,p}(\mathbb R^N)$ for all $\lambda>\lambda^*$.  Moreover, any solution $u\in D^{2,p}(\mathbb R^N)\setminus \{0\}$ satisfies
\[\|u\|\leq \left[\frac{\la \|f\|_{\frac{p^*}{p^*-r}}}{S^{\frac{r}{p}}\left(b- S_{H,L}^{-\frac{2p_{\mu}^{*}}{p}}\right)}\right]^{\frac{1}{p\theta-r}}.\]
\end{theorem}
\begin{remark} The case when $p<r<p\theta$ with $ p\theta<2p^*_\mu$ is an open case due to lack of minimizers.
    
\end{remark}
\begin{remark}  One can generalized these result to the following $p$-biharmonic Kirchhoff equation with critical Stein-Weiss type nonlinearity  
		$$ 
		\begin{array}{lr}
			\left(a+b\left(\int_{\mathbb R^N}|\Delta u|^p dx\right)^{\theta-1}\right) \Delta_{p}^{2}u = \alpha \ds\left(\int_{\mathbb R^N}\frac{|u(y)|^{p_{\mu}^{*}}}{|y|^{\ba}|x-y|^{\mu}}dy \right)\frac{|u(x)|^{p_{\mu}^{*}-2}}{|x|^{\ba}} u + \la f(x) |u|^{r-2} u \; \text{in}\; \mathbb R^N,
		\end{array}
		$$
		where $a\geq 0$, $b\geq 0$, $0<\mu+2\ba<N$, $N>2p$, $p\geq 2$, $\theta\geq1$ and  $\al$, $\la$ are  real positive parameters, $p_{\mu,\ba}^{*}= \frac{p(2N-2\ba-\mu)}{2(N-2p)}$ is the upper critical exponent in the sense of Hardy-Littlewood-Sobolev inequality. The function $f \in L^{t}(\mathbb R^N)$ with $t= \frac{p^{*}}{(p^* -r)}$ if $p<r<p^*:=\frac{Np}{N-2p}$ and $t=\infty$ if $r\geq p^{*}$.
\end{remark}

    
\noi{\bf Organization of article:} In section 2, we enlist the variational framework and preliminary results.  We demonstrate the proof of the concentration compactness principle for the $p$-biharmonic operator with critical Choquard-type nonlinearity. In section 3, we discuss the fundamental results for Palais-Smale sequence, for different ranges of $r$, in both the cases $p\theta< 2p^*_\mu$ and  $p \theta\geq 2p^*_\mu$. In section 4, we consider the case $p \theta\geq 2p^*_\mu$ and prove the Theorems \ref{th3}, \ref{th4}. The case $p\theta< 2p^*_\mu$ is discussed section 5, which is further divided into three subsections $5.1$, $5.2$ and $5.3$ followed by the proofs of Theorems \ref{th1}, \ref{th12}, \ref{th13} respectively.

\section{Variational framework and Preliminary results}
In this section we recall the main notations and tools that will be needed in the sequel. Define the space 
  $D^{2,p}(\mb R^N)=\left\{u\in L^{p^*}(\mb R^N): \int_{\mb R^N}|\Delta u|^p dx <\infty\right\},$  equipped with the norm \[\|u\| :=\|u\|_{D^{2,p}(\mb R^N)}=\left(\int_{\mb R^N}|\Delta u|^p dx\right)^{\frac 1p}. \] is a Banach space. We define  \(S\) to be the best Sobolev constant for the embedding of $D^{2,p}(\mathbb R^N)$ into $L^{p^*}(\mathbb R^N)$ with $p^*=\frac{Np}{N-2p}$ as
	\begin{align}\label{sr}
 \ds S:= \inf_{u\in D^{2,p}(\mathbb R^N)  \setminus \{0\}} \frac{\int_{\Om}|\Delta u|^p dx}{(\int_{\Om}|u(x)|^{p^*} dx)^{\frac{p}{p^*}}}.
 \end{align}
 \noi To handle the convolution-type nonlinearity, we need the well-known Hardy-Littlewood-Sobolev inequality, which is stated as:
\begin{proposition}(Hardy-Littlewood-Sobolev inequality \cite[Theorem 4.3]{r15})
		Let $r$, $s>1$ and $0<\mu<N $ with $1/r+\mu/N+1/s=2$, $g \in L^r(\mathbb R^N)$ and $h \in L^s(\mathbb R^N)$. Then there exists a sharp constant $C(\mu,N,r,s)$, independent of $g,$ $h$ such that
		\begin{equation}\label{h1}
			\int_{\mathbb R^N}\int_{\mathbb R^N} \frac{g(x)h(y)}{|x-y|^{\mu}}dxdy \leq C(\mu,N,r,s)\|g\|_{L^r(\mathbb R^N)}\|h\|_{L^s(\mathbb R^N)}.
		\end{equation}
\end{proposition}

    \noi	Using the Hardy-Littlewood-Sobolev inequality \eqref{h1}, for \(s=r=t\),  the integral is
	\[\int_{\mathbb R^N}\int_{\mathbb R^N} \frac{|u(x)|^{t}|u(y)|^{t}}{|x-y|^{\mu}} dx dy\]
	is well-defined if \(|u|^t  \in L^{q}(\mathbb R^N)\) for some \(q >1\) satisfying \(\frac{2}{q} + \frac{\mu}{N}= 2\).
	\noi For \(u \in W^{2,p}(\mathbb R^N)\), by the Sobolev embedding theorem, we have \(p\leq tq \leq \frac{Np}{N-2p}\). Thus
	\[\frac{p(2N-\mu)}{2N} \leq t \leq \frac{p(2N-\mu)}{2(N-2p)}\]
	In this sense, we call \(p_{*\mu} = \frac{p(2N-\mu)}{2N}\) the lower critical exponent and \(p^{*}_{\mu} = \frac{p(2N-\mu)}{2(N-2p)}\) the upper critical exponent in the sense of the Hardy Littlewood-Sobolev inequality. Infact, for $u\in D^{2,p}(\mathbb R^N)$, the Hardy Liitlewood Sobolev inequality for the upper critical exponent leads to the following
	\[\left(\int_{\mathbb R^N}\int_{\mathbb R^N} \frac{|u(x)|^{p_{\mu}^{*}}|u(y)|^{p_{\mu}^{*}}}{|x-y|^{\mu}} dx dy \right)^{\frac{1}{p^{*}_{\mu}}} \leq  C(N,\mu)^{\frac{1}{p_{\mu}^*}} |u|_{p^*}^{2}. \]
    \noi We define $S_{H, L}$ to be the best constant as
	\begin{align}\label{n3}
		S_{H, L} := \inf_{u\in D^{2,p}(\mathbb R^N)\setminus \{0\}} \frac{\int_{\mathbb R^N}|\Delta u|^p dx}{\left(\int_{\mathbb R^N}\int_{\mathbb R^N}\frac{|u(x)|^{p_\mu^*}|u(y)|^{p_\mu^*}}{|x-y|^{\mu}}dxdy\right)^{\frac{p}{2p_\mu^*}}}.
	\end{align}
	
\noi Then it is easy to see that
	\[{C(N,\mu)^{\frac{p}{2p_{\mu}^*}}}  S_{H,L}\geq {S}>0,\]
	where $S$ is defined in \eqref{sr}.

 \begin{lemma}
Let $N>2p$, $0<\mu<N$ and $\{u_k\}$ be a bounded sequence in $L^{p^*}(\mathbb R^N)$, then the following result holds as $k\rightarrow \infty$
{\small\begin{align*}
		\left(\int_{\mathbb R^N}\frac{|u_k(y)|^{p^*_\mu}|u_k(x)|^{p^*_\mu-2}}{|x-y|^{\mu}}dy\right)u_k(x) \rightharpoonup \left(\int_{\mathbb R^N}\frac{|u(y)|^{p^*_\mu}|u(x)|^{p^*_\mu-2}}{|x-y|^{\mu}}dy\right)u(x)\;\text{weakly in}\; L^{\frac{Np}{Np-N+2p}}(\mathbb R^N).
\end{align*}}
\end{lemma}
\begin{proof}
Let $\{u_k\}$ be a bounded sequence in $L^{p^*}(\mathbb R^N)$, then one can easily verify that
\begin{equation}\label{n60}
	\begin{aligned}
		|u_k|^{p^*_\mu}\rightharpoonup |u|^{p^*_\mu}\;\text{weakly in}\; L^{\frac{2N}{2N-\mu}}(\mathbb R^N),\\
		|u_k|^{p^*_\mu-2}u_k\rightharpoonup |u|^{p^*_\mu-2}u\;\text{weakly in}\; L^{\frac{2Np}{2Np-\mu p-2N+4p}}(\mathbb R^N),
	\end{aligned}
\end{equation}
as $k\rightarrow \infty$. The Riesz potential defines a continuous map from $L^{\frac{2N}{2N-\mu}}(\mathbb R^N)$ to $L^{\frac{2N}{\mu}}(\mathbb R^N)$ by Hardy-Littlewood-Sobolev inequality. Thus, we have
\begin{align}\label{n61}
	\int_{\mathbb R^N}\frac{|u_k(y)|^{p^*_\mu}}{|x-y|^{\mu}}dy \rightharpoonup \int_{\mathbb R^N}\frac{|u(y)|^{p^*_\mu}}{|x-y|^{\mu}}dy\;\text{weakly in}\; L^{\frac{2N}{\mu}}(\mathbb R^N),
\end{align}
as $k\rightarrow \infty$. Then, on combining \eqref{n60} and \eqref{n61}, we obtain
{\small\begin{align*}
		\left(\int_{\mathbb R^N}\frac{|u_k(y)|^{p^*_\mu}|u_k(x)|^{p^*_\mu-2}}{|x-y|^{\mu}}dy\right)u_k(x) \rightharpoonup \left(\int_{\mathbb R^N}\frac{|u(y)|^{p^*_\mu}|u(x)|^{p^*_\mu-2}}{|x-y|^{\mu}}dy\right)u(x)\;\text{weakly in}\; L^{\frac{Np}{Np-N+2p}}(\mathbb R^N)
\end{align*}}
as $k\rightarrow \infty$, which is the required result.
\end{proof}
In order to prove the Palais-Smale condition, we need the following Lemma which is inspired by the Br{\'e}zis-Lieb convergence Lemma (see \cite{Brezislieb}).
\begin{lemma}
Let $N>2p$, $0<\mu<N$ and $\{u_k\}$ be a bounded sequence in $L^{\frac{pN}{N-2p}}(\mathbb R^N)$. If $u_k\rightarrow u$ a.e. in $\mathbb R^N$ as $k \rightarrow \infty$, then
{\small\begin{align*}
\lim_{k \rightarrow \infty}\left(\int_{\mathbb R^N}(|x|^{-\mu}*|u_k|^{p^{*}_{\mu}})|u_k|^{p^{*}_{\mu}} - \int_{\mathbb R^N}(|x|^{-\mu}*|u_k -u|^{p^{*}_{\mu}})|u_k-u|^{p^{*}_{\mu}}\right) = \int_{\mathbb R^N}(|x|^{-\mu}*|u|^{p^{*}_{\mu}})|u|^{p^{*}_{\mu}}.
\end{align*}}
\end{lemma}
\begin{proof} The proof is similar to the proof of the Br{\'e}zis-Lieb Lemma (see \cite{Brezislieb}) or Lemma 2.2 \cite{gao}. But for completeness, we give the detail. Consider
\begin{align}\label{1}
\int_{\mathbb R^N}&(|x|^{-\mu}*|u_k|^{p^{*}_{\mu}})|u_k|^{p^{*}_{\mu}} - \int_{\mathbb R^N}(|x|^{-\mu}*|u_k -u|^{p^{*}_{\mu}})|u_k-u|^{p^{*}_{\mu}} \notag\\
&= \int_{\mathbb R^N}(|x|^{-\mu}*(|u_k|^{p^{*}_{\mu}}-|u_k-u|^{p^{*}_{\mu}}))(|u_k|^{p^{*}_{\mu}}-|u_k-u|^{p^{*}_{\mu}})\notag \\
&\quad\quad+ 2\int_{\mathbb R^N}(|x|^{-\mu}*(|u_k|^{p^{*}_{\mu}}-|u_k-u|^{p^{*}_{\mu}}))|u_k-u|^{p^{*}_{\mu}}.
\end{align}
Now by using \cite{moroz2013} $(\text{lemma}\; 2.5)$, for $q=p^*_{\mu}=\frac{p(2N-\mu)}{2(N-2p)}$ and $r= \frac{pN}{2N-\mu}p^{*}_{\mu}$, then we obtain
\begin{align}\label{m1}
|u_k|^{p^{*}_{\mu}} - |u_k-u|^{p^{*}_{\mu}} \rightarrow |u|^{p^{*}_{\mu}} \;\;\text{in}\;\; L^{\frac {2N}{2N-\mu}}(\mathbb R^N)\;\;\text{as}\;\; k\rightarrow \infty.
\end{align}
Also the Hardy-Littlewood-Sobolev inequality implies that
\begin{align}\label{m2}
|x|^{-\mu}*(|u_k|^{p^{*}_{\mu}} - |u_k-u|^{p^{*}_{\mu}}) \rightarrow |x|^{-\mu}*|u|^{p^{*}_{\mu}} \;\;\text{in}\;\; L^{\frac {2N}{\mu}}(\mathbb R^N)\;\;\text{as}\;\; k\rightarrow \infty.
\end{align}
 Hence, with the help of \cite{willem} $(\text{Prop.}\; 5.4.7)$, we obtain $|u_k-u|^{p^{*}_{\mu}}\rightharpoonup 0$ weakly in $L^{\frac {pN}{2N-\mu}}(\mathbb R^N)$ as $k \rightarrow \infty$. So using this together with \eqref{m1}, \eqref{m2}, in \eqref{1}, we obtain the required result. 
\end{proof}
We  recall the concentration compactness Lemmas given by P. L. Lions \cite{r2, r3}.

\begin{lemma}\label{y12}
Let $\{u_k\}$ be a bounded sequence in $D^{2, p}(\mathbb R^N)$ converging weakly and a.e. to $u\in D^{2, p}(\mathbb R^N)$ such that $|\Delta u_k|^p\rightharpoonup \omega$, $|u_k|^{p^*}\rightharpoonup \zeta$ in the sense of measure. Then, for at most countable set $J$, there exist families of distinct points $\{\omega_i: i\in J\}$ and $\{\zeta_i: i\in J\}$ in $\mathbb R^N$ satisfying
\begin{align*}
	\zeta = |u|^{p^*} + \sum_{i \in J} \zeta_i \delta_{z_i},\; \zeta_i>0,\\
	\omega \geq  |\Delta u|^p + \sum_{i \in J} \omega_i \delta_{z_i},\;\ba_i>0,\\
	S \zeta_i^{\frac{p}{p^*}} \leq  \omega_i, \; \forall\; i\in J,
\end{align*}
where $\zeta$, $\omega$ are bounded and non negative measures on $\mathbb R^N$ and $\delta_{z_i}$ is the Dirac mass at $z_i$.  In particular, \(\ds\sum_{i\in J} (\zeta_i)^{\frac{p}{p^*}}<\infty\).
\end{lemma}

\begin{lemma}{\label{1e0}}
Let \(\{u_k\}\subset  D^{2,p}(\mathbb R^N)\) be a sequence in  Lemma \ref{y12} and defined
\[ \omega_{\infty}:=\lim_{R\ra\infty}\limsup_{ k\ra\infty} \int_{|x|>R} |\Delta u_k|^p dx, \quad  \zeta_{\infty}= \lim_{R\ra\infty}\limsup_{ k\ra\infty} \int_{|x|>R} |u_k|^{p^*} dx \]
Then it follows that
\[S \zeta_{\infty}^{p/p^*} \leq \omega_{\infty}.\]
and
\end{lemma}

\noi Now, we prove the following concentration compactness Lemma  for our problem $(\mc P_{\al,\la})$.
\begin{lemma}
Let \(N>2p\), \(0<\mu<N\), and \(1\leq p\leq \frac{p(2N-\mu)}{2(N-2p)}\). If $\{u_k\}$ is a bounded sequence in \(D^{2,p}(\mathbb  R^N)\) converges weakly, to $u$ as \(k\ra \infty\) and such that \(|u_k|^{p^*} \rightharpoonup \zeta\) and \(|\Delta u_k|^p \rightharpoonup \omega\) in the sense of measure. Assume that
\[\left(\int_{\mathbb R^N}\frac{|u_{k}(y)|^{p^{*}_{\mu}}}{|x-y|^{\mu}}dy\right)|u_{k}(x)|^{p^{*}_{\mu}} \rightharpoonup \nu\]
weakly in the sense of measure where \(\nu\) is a bounded positive measure on \(\mathbb R^N\) and define
\[\omega_{\infty} := \lim_{R\ra \infty} \limsup_{k\ra \infty} \int_{|x|\geq R} |\Delta u_k|^p dx, \quad \zeta_{\infty}:= \lim_{R\ra \infty} \limsup_{k\ra \infty} \int_{|x|\geq R} |u_k|^{p^*} dx. \]
\[ \nu_{\infty} := \lim_{R\ra \infty} \limsup_{k\ra \infty} \int_{|x|\geq R}\left(\int_{\mathbb R^N}\frac{|u_{k}(y)|^{p^{*}_{\mu}}}{|x-y|^{\mu}}dy\right)|u_{k}(x)|^{p^{*}_{\mu}}dx.\]
Then there exists a countable sequence of points \(\{z_i\}_{i\in J}\subset\mathbb R^N\) and families of positive numbers $\{\omega_i : i\in J\}$, $\{\zeta_i: i\in J\}$ and $\{\nu_i: i\in J\}$ such that
\begin{align}
	\nu & = \left(\int_{\mathbb R^N}\frac{|u(y)|^{p^{*}_{\mu}}}{|x-y|^{\mu}}dy\right) |u(x)|^{p^{*}_{\mu}} +\sum_{i\in J} \nu_{i} \delta_{z_i}, \quad \sum_{i\in J} \nu_i^{\frac{1}{p^{*}_{\mu}}} < \infty,\label{c1}\\
	\omega &\geq |\Delta u|^p + \sum_{i\in J} \omega_{i} \delta_{z_i},\label{l2}\\
	\zeta &\geq  |u|^{p^{*}} + \sum_{i\in J} \zeta_{i} \delta_{z_i},\label{l3}
\end{align}
and
\begin{align}
	S_{H,L} \nu_{i}^{\frac{p}{2p^{*}_{\mu}}}\leq \omega_i, \;\mbox{and} \; \nu_{i}^{\frac{N}{2N-\mu}} \leq C(N,\mu)^{\frac{N}{2N-\mu}} \zeta_i,\label{l4}
\end{align}
where \(\delta_x\) is the Dirac-mass of mass \(1\) concentrated at \(x \in\mathbb R^N\).

\noi For the energy at infinity, we have
\[\limsup_{k\ra \infty} \int_{\mathbb R^N } |\Delta u_k|^p dx= \omega_{\infty} + \int_{\mathbb R^N} d\omega, \quad \limsup_{k\ra \infty} \int_{\mathbb R^N} |u_k|^{p^*} dx = \zeta_{\infty} + \int_{\mathbb R^N} d\zeta, \]
\[\limsup_{k\ra \infty} \int_{\mathbb R^N}\int_{\mathbb R^N}\frac{|u_k(y)|^{p^{*}_{\mu}}|u_k(y)|^{p^{*}_{\mu}}}{|x-y|^{\mu}}   dxdy = \nu_{\infty}+ \int_{\mathbb R^N} d\nu,\]
and
\[C(N,\mu)^{-\frac{2N}{2N-\mu}} \nu_{\infty}^{\frac{2N}{2N-\mu}} \leq \zeta_{\infty} \left(\int_{\mathbb R^N} d\zeta + \zeta_{\infty}\right), \quad S^{p} C(N,\mu)^{-\frac{p}{p_{\mu}^*}}\nu_{\infty}^{\frac{p}{p_{\mu}^*}} \leq \omega_{\infty} \left(\int_{\mathbb R^N} d\omega + \omega_{\infty}\right).\]
\end{lemma}

\begin{proof}
Let $v_k=u_k-u$. Then $\{v_k\}$ converging weakly to $0$ in $D^{2, p}(\mathbb R^N)$ and $v_k(x)\rightarrow 0$ a.e. in $\mathbb R^N$ as the bounded sequence $\{u_k\}$ converging weakly to $u$ in $D^{2, p}(\mathbb R^N)$. Lemma \ref{y12} yields
\begin{align*}
	|\De v_k|^p\rightharpoonup &\; \tau_1:=\omega-|\Delta u|^p,\\
	|v_k|^{p^*}\rightharpoonup &\; \tau_2:=\zeta-|u|^{p^*},\\
	\left(\int_{\mathbb R^N}\frac{|v_k(y)|^{p_\mu^*}}{|x-y|^{\mu}}dy\right)|v_k(x)|^{p_\mu^*}\rightharpoonup &\;\tau_3:=\nu - \left(\int_{\mathbb R^N}\frac{|u(y)|^{p_{\mu}^*}}{|x-y|^{\mu}}dy\right)|u(x)|^{p_{\mu}^*}.
\end{align*}
Firstly, we show that for every $\phi \in C_c^{\infty}(\mathbb R^N)$,
\begin{align}\label{n32}
	\left|\int_{\mathbb R^N}\left(|x|^{-\mu}*|\phi v_k(x)|^{p_{\mu}^*}\right)|\phi v_k(x)|^{p_{\mu}^*}dx - \int_{\mathbb R^N}\left(|x|^{-\mu}*| v_k(x)|^{p_\mu^*}\right)|\phi(x)|^{p_\mu^*}|\phi v_k(x)|^{p_{\mu}^*}dx\right|\rightarrow 0.
\end{align}
For this, we denote \[\Psi_k(x):=\left[\left(|x|^{-\mu}*|\phi v_k(x)|^{p_\mu^*}\right) - \left(|x|^{-\mu}*| v_k(x)|^{p_\mu^*}\right)|\phi(x)|^{p_\mu^*}\right]|\phi v_k(x)|^{p_\mu^*}.\] 
Since $\phi \in C_c^{\infty}(\mathbb R^N)$, we have for every $\delta >0$, there exists $K>0$ such that
\begin{align}\label{n31}
	\int_{|x|\geq K}|\Psi_k(x)|dx < \delta\;\;\forall\;k\geq 1.
\end{align}
As we know that Riesz potential defines a linear operator and using $v_k(x)\rightarrow 0$ a.e. in $\mathbb R^N$, so we obtain
\begin{align*}
	\int_{\mathbb R^N}\frac{|v_k(y)|^{p_\mu^*}}{|x-y|^{\mu}}dy\rightarrow 0\;\text{a.e. in}\;\mathbb R^N.
\end{align*}
Thus $\Psi_k(x)\rightarrow 0$ a.e. in $\mathbb R^N$. We note that
\begin{align*}
	\Psi_k(x)
	=& \int_{\mathbb R^N}\frac{\left(|\phi(y)|^{p_\mu^*}-|\phi(x)|^{p_\mu^*}\right)}{|x-y|^{\mu}}|v_k(y)|^{p_\mu^*}dy|\phi v_k(x)|^{p_\mu^*}
	:=\int_{\mathbb R^N}\Phi(x, y)|v_k(y)|^{p_\mu^*}dy|\phi v_k(x)|^{p_\mu^*},
\end{align*}
where $\Phi(x, y)=\frac{|\phi(y)|^{p_\mu^*}-|\phi(x)|^{p_\mu^*}}{|x-y|^{\mu}}$. Moreover, for almost all $x$, there exists some $R>0$ large enough such that
\begin{align*}
	\int_{\mathbb R^N}\Phi(x, y)|v_k(y)|^{p_\mu^*}dy
	=&\int_{|y|\geq R}\Phi(x, y)|v_k(y)|^{p_\mu^*}dy - |\phi(x)|^{p_\mu^*}\int_{|y|\geq R}\frac{|v_k(y)|^{p_\mu^*}}{|x-y|^{\mu}}dy.
\end{align*}
Using the mean value theorem and the fact that $p\geq 2$,  $\Phi(x, y)\in L^q(B_R)$ for each $x$, where $q<\frac{N}{\mu-1}$ if $\mu>1$, $q\leq +\infty$ if $0<\mu \leq 1$. With the help of Young's inequality, there exists $t>\frac{2N}{\mu}$ such that
\begin{align*}
	\left(\int_{B_K}\left(\int_{B_R}\Phi(x, y)|v_k(y)|^{p_\mu^*}dy\right)^t dx\right)^{\frac 1t}\leq L_{\phi}\|\Phi(x, y)\|_q\||v_k|^{p_\mu^*}\|_{\frac{2N}{2N-\mu}}\leq L^{\prime}_{\phi},
\end{align*}
where $K$ is same as in $\eqref{n31}$. Moreover, one can easily see that for $R>0$ large enough
\begin{align*}
	\left(\int_{B_K}\left(|\phi(x)|^{p_\mu^*}\int_{|y|\geq R}\frac{|v_k(y)|^{p_\mu^*}}{|x-y|^{\mu}}dy\right)^t dx\right)^{\frac 1t}\leq L,
\end{align*}
and so, we have
\begin{align*}
	\left(\int_{B_K}\left(\int_{\mathbb R^N}\Phi(x, y)|v_k(y)|^{p_\mu^*}dy\right)^t dx\right)^{\frac 1t}\leq L^{\prime\prime}_{\phi}.
\end{align*}
Thus for $s>0$ small enough, we obtain
\begin{align*}
	\int_{B_K}|\Psi_k(x)|^{1+s}dx\leq\left(\int_{B_K}\left(\int_{\mathbb R^N}\Phi(x, y)|v_k(y)|^{p_\mu^*}dy\right)^t dx\right)^{\frac 1t}\left(\int_{B_K}|\phi v_k|^{p^*}dx\right)^{\frac{p_\mu^*}{p^*}}\leq L^{\prime\prime}_{\phi}.
\end{align*}
Using this together with $\Psi _k(x)\rightarrow 0$ a.e. in $\mathbb R^N$, we have
\begin{align*}
	\int_{B_K}|\Psi_k(x)|dx \rightarrow 0\;\text{as}\;k\rightarrow\infty.
\end{align*}
Combining this with \eqref{n31}, we have
\begin{align*}
	\int_{\mathbb R^N}|\Psi_k(x)|dx \rightarrow 0\;\text{as}\;k\rightarrow\infty.
\end{align*}
Now, for every $\phi \in C_c^{\infty}(\mathbb R^N)$, by Hardy-Littlewood-Sobolev inequality, we obtain
\begin{align*}
	\int_{\mathbb R^N}\left(\int_{\mathbb R^N}\frac{|\phi v_k(y)|^{p_\mu^*}}{|x-y|^{\mu}}dy\right)|\phi v_k(x)|^{p_\mu^*} dx \leq C(N, \mu)\|\phi v_k\|_{p^*}^{2p_\mu^*}.
\end{align*}
Equation \eqref{n32} yields
\begin{align*}
	\int_{\mathbb R^N}|\phi(x)|^{2p_\mu^*}\left(\int_{\mathbb R^N}\frac{|v_k(y)|^{p_\mu^*}}{|x-y|^{\mu}}dy\right)|v_k(x)|^{p_\mu^*} dx \leq C(N, \mu)\|\phi v_k\|_{p^*}^{2p_\mu^*}.
\end{align*}
On taking the limit as $k\rightarrow \infty$, we obtain
\begin{align}\label{n33}
	\int_{\mathbb R^N}|\phi(x)|^{2p_\mu^*}d\tau_3\leq C(N, \mu)\left(\int_{\mathbb R^N}|\phi|^{p^*}d\tau_2\right)^{\frac{2p_\mu^*}{p^*}}.
\end{align}
Employing Lemma $1.2$ in \cite{r2}, one can directly obtain \eqref{l3}.
\\ Further, let $\phi=\chi_{\{z_i\}}$, $i\in J$ and using this in \eqref{n33}, we obtain
\begin{align*}
	\nu_i^{\frac{p^*}{2p_\mu^*}}\leq \left(C(N, \mu)\right)^{\frac{p^*}{2p_\mu^*}}\zeta_i,\;\forall \;i\in J.
\end{align*}
Definition of $S_{H, L}$ yields
\begin{align*}
	\left(\int_{\mathbb R^N}\left(\int_{\mathbb R^N}\frac{|\phi v_k(y)|^{p_\mu^*}}{|x-y|^{\mu}}dy\right)|\phi v_k(x)|^{p_\mu^*}dx\right)^{\frac{p}{2p_\mu^*}}S_{H, L}\leq \int_{\mathbb R^N}|\De(\phi v_k)|^p dx.
\end{align*}
Also equation \eqref{n32} and $v_k\rightarrow 0$ in $L^p_{loc}(\mathbb R^N)$ give
\begin{align*}
	\left(\int_{\mathbb R^N}|\phi(x)|^{2p_\mu^*}\left(\int_{\mathbb R^N}\frac{|v_k(y)|^{p_\mu^*}}{|x-y|^{\mu}}dy\right)| v_k(x)|^{p_\mu^*}dx\right)^{\frac{p}{2p_\mu^*}}S_{H, L}\leq \int_{\mathbb R^N}|\phi|^p|\De v_k|^p dx + o(1).
\end{align*}
On passing the limit as $k \rightarrow \infty$, we have
\begin{align}\label{n34}
	\left(\int_{\mathbb R^N}|\phi(x)|^{2p_\mu^*}d \tau_3\right)^{\frac{p}{2p_\mu^*}}S_{H, L}\leq \int_{\mathbb R^N}|\phi|^p d\tau_1.
\end{align}
Let $\phi=\chi_{\{z_i\}}$, $i\in J$ and applying this in \eqref{n34}, we have
\begin{align*}
	S_{H, L}\nu_i^{\frac{p}{2p_\mu^*}}\leq \omega_i,\;\forall \;i\in J.
\end{align*}
This completes the proof of \eqref{c1} and \eqref{l4}.

\noi Now, we prove the possible loss of mass at infinity. For $R>1$, let $\psi_{R}\in C^{\infty}(\mathbb R^N)$ be such that $\psi_{R}=1$ for $|x|> R+1$, $\psi_{R}(x)=0$ for $|x|<R$ and $0\leq \psi_{R}(x) \leq 1$ on $\mathbb R^N$. For every $R$, we have
	
    {\small \begin{align*}
		\limsup_{k\ra \infty} \int_{\mathbb R^N} &\int_{\mathbb R^N}\frac{|u_{k}(y)|^{p^{*}_{\mu}}|u_{k}(x)|^{p^{*}_{\mu}}}{|x-y|^{\mu}} dy dx\\
		&= 	\limsup_{k\ra \infty} \left( \int_{\mathbb R^N} \int_{\mathbb R^N}\frac{|u_{k}(y)|^{p^{*}_{\mu}}|u_{k}(x)|^{p^{*}_{\mu}}\psi_{R}(x)}{|x-y|^{\mu}} dy dx + 	 \int_{\mathbb R^N} \int_{\mathbb R^N}\frac{|u_{k}(y)|^{p^{*}_{\mu}}|u_{k}(x)|^{p^{*}_{\mu}}(1-\psi_{R}(x))}{|x-y|^{\mu}}  dy dx  \right)\\
		&= 	\limsup_{k\ra \infty} \int_{\mathbb R^N} \int_{\mathbb R^N}\frac{|u_{k}(y)|^{p^{*}_{\mu}}|u_{k}(x)|^{p^{*}_{\mu}}\psi_{R}(x)}{|x-y|^{\mu}} dy dx + 	\int_{\mathbb R^N} (1- \psi_{R}) d\nu.
	\end{align*}}
	Taking $R\ra \infty$, by Lebesgue's  dominated convergent theorem, we deduce
	\[	\limsup_{k\ra \infty} \int_{\mathbb R^N} \int_{\mathbb R^N}\frac{|u_{k}(y)|^{p^{*}_{\mu}}|u_{k}(x)|^{p^{*}_{ \mu}} }{|x-y|^{\mu} }  dy dx= \nu_\infty +\int_{\mb R^N} d\nu. \]
By the Hardy-Littlewood-Sobolev inequatlity \eqref{h1}, we obtain
	\begin{align*}
		\nu_{\infty} &=\lim_{R\ra\infty} \limsup_{k\ra \infty} \int_{\mathbb R^N} \left(\int_{\mathbb R^N}\frac{|u_k(y)|^{p^{*}_{\mu}}}{|x-y|^{\mu}} dy \right) {|\psi_{R}u_k(x)|^{p^{*}_{\mu}}}  dx \\
		&\leq C(N,\mu) \lim_{R\ra\infty} \limsup_{k\ra \infty}  \left(\int_{\mathbb R^N} |u_k|^{p^*} dx \int_{\mathbb R^N} |\psi_{R} u_k|^{p^*} dx \right)^{\frac{2N-\mu}{2N}}\\
		&=  C(N,\mu)  \left((\zeta_{\infty}+ \int_{\mathbb R^N} d\zeta) \zeta_{\infty} \right)^{\frac{2N-\mu}{2N}}.
	\end{align*}
	This implies
	\[C(N,\mu)^{-\frac{2N}{2N-\mu}} \nu_{\infty}^{\frac{2N}{2N-\mu}} \leq \zeta_{\infty} \left(\int_{\mathbb R^N} d\zeta + \zeta_{\infty}\right).\]
	Similarly, by the Hardy-Littlewood-Sobolev inequatlity \eqref{h1}, we obtain
	\begin{align*}
		\nu_{\infty} &=\lim_{R\ra\infty} \limsup_{k\ra \infty} \int_{\mathbb R^N} \left(\int_{\mathbb R^N}\frac{|u_k(y)|^{p^{*}_{\mu}}}{|x-y|^{\mu}} dy \right) {|\psi_{R}u_k(x)|^{p^{*}_{\mu}}}  dx \\
		&\leq C(N,\mu) \lim_{R\ra\infty} \limsup_{k\ra \infty}  \left(\int_{\mathbb R^N} |u_k|^{p^*} dx \int_{\mathbb R^N} |\psi_{R} u_k|^{p^*} dx \right)^{\frac{2N-\mu}{2N}}\\
		&\leq C(N,\mu) S^{-p_{\mu}^{*}} \lim_{R\ra\infty} \limsup_{k\ra \infty}  \left(\int_{\mathbb R^N} |\Delta u_k|^{p} dx \int_{\mathbb R^N} |\De(\psi_{R} u_k)|^{p} dx \right)^{\frac{p^{*}_{\mu}}{p}}\\
		&=  C(N, \mu) S^{-p_{\mu}^{*}} \left((\omega_{\infty}+ \int_{\mathbb R^N} d\omega)\omega_{\infty} \right)^{\frac{p^{*}_{\mu}}{p}}.
	\end{align*}
	This implies
	\[S^{p} C(N,\mu)^{-\frac{p}{p_{\mu}^{*}}} \nu_{\infty}^{\frac{p}{p_{\mu}^{*}}} \leq \omega_{\infty} \left(\int_{\mathbb R^N} d\omega + \omega_{\infty}\right).\]
	This completes the proof. 
\end{proof}

\noi	We state the general version of the mountain pass Lemma which will be used to prove the Theorem.
	\begin{theorem}		Let $I$ be a functional on a Banach space \(X\) and \(I\in C^{1}(E, \mathbb R)\). If there exists \(\alpha\), \(\rho\) such that
		\begin{enumerate}
			\item \(I(u)\geq \al\), \(u\in X\) with \(\|u\|= \rho\);
			\item \(I(u)=0\) and $I(e)<0$ for some $e\in E$ with $\|e\|>\rho$.
		\end{enumerate}
		Define
		\[\Gamma_j=\{\ga\in C([0,1], X): \ga(0)=0, \ga(1)=e\}\]
		and
		\[c= \inf_{\ga\in \Gamma} \max_{t\in[0,1]} I(\ga(t)).\]
		
		\noi Then there exists a sequence $\{u_k\}_k \subset X$ such that $I(u_k)\ra c$ and $I^{\prime}(u_k) \ra 0$ in $X^{\prime}$.
	\end{theorem}

 	
	\begin{definition}
		Let $I: X\rightarrow \mathbb R$ be a $C^{1}$ functional on a Banach space $X$.
		\begin{enumerate}
			\item For $c\in \mb R$, a sequence $\{u_k\}\subset X$ is a $(PS)_c$ (Palais-Smale sequence at level $c$)  in $X$ for $I$ if $I(u_k)=c +o_{k}(1)$ and $I^{\prime}(u_k) \rightarrow 0$ in $X^{-1}$ as $k\rightarrow \infty.$
			\item We say $I$ satisfies $(PS)_{c}$ condition if for any Palais-Smale sequence $\{u_k\}$ in $X$ for $I$ has a convergent subsequence in X.
		\end{enumerate}
	\end{definition}
	
	


	\noi We define the energy functional $ \mc E_{\al,\la}$ corresponding to the problem $(\mc P_{\al, \la})$ as
	\begin{align*}
		\mc E_{\al,\la}(u) = \frac{a}{p} \|u\|^{p}+ \frac{b}{p\theta} \|u\|^{p\theta} - \frac{\al}{2 p_{\mu}^*}\int_{\mathbb R^N}\int_{\mathbb R^N}\frac{|u(x)|^{p_{\mu}^*}|u(y)|^{p_{\mu}^*}}{|x-y|^{\mu}}dxdy - \frac{\la}{r}\int_{\mathbb R^N}f(x)|u|^r dx.
	\end{align*}
	Then, by Hardy-Littlewood-Sobolev inequality, one can easily see that $\mc E_{\al, \la}\in C^{1}(D^{2,p}(\mathbb R^N), \mb R)$. Moreover, $u$ is a weak solution of the problem $(\mc P_{\al, \la})$ if and only if $u$ is a critical point of the functional $\mc E_{\al,\la}$.
	A function $u \in D^{2,p}(\mathbb R^N)$ is said to be a weak solution of $(\mc P_{\al,\la})$ if, for all $\phi \in D^{2,p}(\mathbb R^N) $,
	{\small\begin{align*}
			(a+ b \|u\|^{p(\theta-1)}) \int_{\mathbb R^N}|\Delta u|^{p-2}\Delta u\De \phi dx &= \la \int_{\mathbb R^N}f(x)|u|^{r-2}u\phi dx \\&\quad+ \al \int\int_{\mb R^{2N}}\frac{|u(y)|^{p_\mu^*}|u(x)|^{p_\mu^*-2}u(x)\phi(x)}{|x-y|^{\mu}} dydx.
	\end{align*}}
	
Throughout the article, for $\al=1$, we denote  the problem $ (\mc P_{\al,\la})$ by $(\mc P_{\la})$ and energy functional $ \mc E_{\al,\la}(u)$ by $\mc E_{\la}(u)$.	
	\section{The Palais-Smale condition}
\noi This section is  divided into two subsections $3.1$ and $3.2$ in which we discuss how the $(PS)_c$ sequence satisfies the Palais-Smale condition in the cases  $p\theta < 2p_{\mu}^*$ and $p\theta \geq 2p_{\mu}^*$ respectively.
    \subsection{Case 1: \texorpdfstring{$p\theta < 2p_{\mu}^*$ }{}}
    \noi In this subsection, we use concentration compactness principle which we proved in section $2$ to show the $(PS)_c$ condition for different range of $r$.
	\begin{lemma}\label{l1}
		Let  $p\theta < 2 p^{*}_{\mu}$ and $1 <r<p^*$. Then any $(PS)_c$ sequence $\{u_k\}$ of $\mc E_{\al,\la}$ is bounded in \(D^{2,p}(\mathbb R^N)\).
	\end{lemma}
	
	\begin{proof}
		Let \(\{u_k\}\) be a $(PS)_c$ sequence in $D^{2,p}(\mathbb R^N)$. Then
		\begin{align}\label{e1}
			\frac{a}{p} \|u_k\|^{p}+ \frac{b}{p\theta} \|u_k\|^{p\theta} - \frac{\alpha}{2p_{\mu}^*}\int_{\mathbb R^N}\int_{\mathbb R^N}\frac{|u_k(x)|^{p_{\mu}^*}|u_k(y)|^{p_{\mu}^*}}{|x-y|^{\mu}}dxdy - \frac{\la}{r}\int_{\mathbb R^N}f(x)|u_k|^{r} dx= c+o_k(1),
		\end{align}
  and for all \(\phi \in C_c^{\infty}(\mb R^N)\),
		\begin{align}\label{e2}
			(a+ b \|u_k\|^{p(\theta-1)}) &\int_{\mathbb R^N}|\Delta u_k|^{p-2}\Delta u_k \De \phi dx -\la \int_{\mathbb R^N}f(x)|u_k|^{r-2}u_{k}\phi dx \notag\\
			& \quad-\mu \int\int_{ \mb R^{2N}}\frac{|u_k(y)|^{p_\mu^*}|u_k(x)|^{p_\mu^*-2}u_{k}(x)\phi(x)}{|x-y|^{\mu}} dxdy = o(1)\|u_k\|.
		\end{align}
		Now using H\"{o}lder inequality and Sobolev embedding Theorem, we can easily deduce that
		\begin{align}\label{e3}
			\int_{\mathbb R^N}f(x)|u_k|^{r} dx \leq S^{-\frac{r}{p}} \|f\|_{\frac{p^{*}}{p^* -r}} \|u_k\|^{r}.
		\end{align}
\noi Case 1: $1< r<p$. Equations \eqref{e1}, \eqref{e2} and \eqref{e3} yield, as $k\ra \infty$
		\begin{align*}
			c+ o_k(1)\|u_k\|&= \mc E_{\al,\la}(u_k) - \frac{1}{2 p_{\mu}^{*}}  \langle \mc E^{\prime}_{\al,\la}(u_k), u_k\rangle \\
			&\geq \left(\frac{1}{p}- \frac{1}{2 p_{\mu}^{*}}\right)a \|u_k\|^p + \left(\frac{1}{p\theta}- \frac{1}{2 p_{\mu}^{*}}\right)b \|u_k\|^{p\theta}- \la\left(\frac{1}{r}- \frac{1}{2 p_{\mu}^{*}}\right) \int f(x)|u_{k}|^{r} dx\\
			&\geq \left(\frac{1}{p}- \frac{1}{2 p_{\mu}^{*}}\right)a \|u_k\|^p - \left(\frac{1}{r}- \frac{1}{2 p_{\mu}^{*}}\right) \la S^{-\frac{r}{p}} \|f\|_{\frac{p^{*}}{p^*-r}} \|u_k\|^{r},
		\end{align*}
		since $p \theta< 2 p^{*}_{\mu}$, and ${1<r<p}$. This implies that $\{u_k\}$ is bounded in $D^{2,p}(\mathbb R^N)$. 

\noi Case 2: $r=p$. Again as in case 1, we have, as $k\ra \infty$
		\begin{align*}
			c+ o_k(1)\|u_k\|&= \mc E_{\al,\la}(u_k) - \frac{1}{2 p_{\mu}^{*}}  \langle \mc E^{\prime}_{\al,\la}(u_k), u_k\rangle \\
			&\geq \left(\frac{1}{p}- \frac{1}{2 p_{\mu}^{*}}\right) \left(a- \la S^{-1} \|f\|_{\frac{p^{*}}{p^*-r}}\right)\|u_k\|^p,
		\end{align*}
	since $0<\la< \frac{a}{\la S^{-1} \|f\|_{\frac{p^{*}}{p^*-r}}}$, $\{u_k\}$ is bounded in $D^{2,p}(\mathbb R^N)$.
 
\noi Case 3:  $p<r< p\theta$. Same as in case 1, we have, as $k\rightarrow\infty$
\begin{align*} 
			c+ o_k(1)\|u_k\|&= \mc E_{\al,\la}(u_k) - \frac{1}{2 p_{\mu}^{*}}  \langle \mc E^{\prime}_{\al,\la}(u_k), u_k\rangle \\
			&\geq \left(\frac{1}{p}- \frac{1}{2 p_{\mu}^{*}}\right)a \|u_k\|^p - \left(\frac{1}{r}- \frac{1}{2 p_{\mu}^{*}}\right) \la S^{-\frac{r}{p}} \|f\|_{\frac{p^{*}}{p^*-r}} \|u_k\|^{r},
		\end{align*}
		since $p \theta< 2 p^{*}_{\mu}$, and ${p<r< p\theta}$. This implies that $\{u_k\}$ is bounded in $D^{2,p}(\mathbb R^N)$.

\noi Case 3: $p\theta \leq r<p^{*}$. Using \eqref{e1}, \eqref{e2} and \eqref{e3},  as $k\ra \infty$, we deduce that
		{\small\begin{align*}
			c+ o_k(1)\|u_k\|&= \mc E_{\al,\la}(u_k) - \frac{1}{r}  \langle \mc E^{\prime}_{\al,\la}(u_k), u_k\rangle \\
			&\geq \left(\frac{1}{p}- \frac{1}{r}\right)a \|u_k\|^p + \left(\frac{1}{p\theta}- \frac{1}{r}\right)b \|u_k\|^{p\theta}+ \left(\frac{1}{r}- \frac{1}{2 p_{\mu}^{*}}\right) \int_{\mb R^{N}}\int_{\mb R^{N}} \frac{|u_k(y)|^{p_{\mu}^{*}}|u_k(x)|^{p_{\mu}^{*}}}{|x-y|^{\mu}}dy dx \\
			&\geq \left(\frac{1}{p}- \frac{1}{r}\right)a \|u_k\|^p,
		\end{align*}}
		using the fact that $p\theta\leq r<p^*< 2 p^{*}_{\mu}$. Therefore, $\{u_k\}$ is bounded in $D^{2,p}(\mathbb R^N).$
  This complete the proof of Lemma.\qed
	\end{proof}
	

	
	\begin{lemma}\label{p11}
		Let \(\{u_k\}\subset D^{2,p}(\mathbb R^N)\) be a Palais-Smale sequence of functional \(\mc E_{\al,\la}\). If $p\theta < 2 p_{\mu}^*$, $ 1 < r < p$ and $c<0$. Then the following two properties holds:
		\begin{enumerate}
			\item[1] For each $\lambda>0$ there exists $\overline{\La} >0$ such that $\mc E_{\al,\lambda}$ satisfies the $(PS)_c$ condition for all $\al \in (0, \overline{\La})$.
			\item[2] For each $\alpha>0$ there exists $\underline{\La} >0$ such that $\mc E_{\al,\lambda}$ satisfies the $(PS)_c$ condition for all $\lambda \in (0, \underline{\La})$.
		\end{enumerate}
		This means there exists a subsequence of \(\{u_k\}\) which converges strongly in \(D^{2,p}(\mathbb R^N)\).
	\end{lemma}
	
	\begin{proof}
		Let $\{u_k\}\subset D^{2,p}(\mathbb R^N)$ be a $(PS)_c$ sequence. Then by Lemma \ref{l1}, $\{u_k\}$ is a bounded in $D^{2,p}(\mathbb R^N)$.  Therefore we can assume that \(u_k \rightharpoonup u\) in $D^{2,p}(\mathbb R^N)$, $u_k \ra u$ a.e in $\mathbb R^N$. 
		
		\noi By Lemma \ref{c1}, there exists at most countable set $J$, sequence of points $\{z_i\}_{i\in J} \subset \mathbb R^N$
		and families of positive numbers $\{\mu_i : i\in J\}$, $\{\zeta_i: i\in J\}$ and $\{\omega_i: i\in J\}$ such that
		\begin{align*}
			\nu & = \left(\int_{\mathbb R^N}\frac{|u(y)|^{p^{*}_{\mu}}}{|x-y|^{\mu}}\right){|u(x)|^{p^{*}_{\mu}}} +\sum_{i\in J} \nu_{i} \delta_{z_i}, \sum_{i\in J} \nu_i^{\frac{1}{p^{*}_{\mu}}} < \infty,\\
			\omega &\geq |\Delta u|^p + \sum_{i\in J} \omega_{i} \delta_{z_i}\\
			\zeta &\geq  |u|^{p^{*}} + \sum_{i\in J} \zeta_{i} \delta_{z_i}
		\end{align*}
		and
		\begin{align*}
			S_{H,L} \nu_{i}^{\frac{p}{2p^{*}_{\mu}}}\leq \omega_i, \;\mbox{and} \; \nu_{i}^{\frac{N}{2N-\mu}} \leq C(N,\mu)^{\frac{N}{2N-\mu}} \zeta_i.
		\end{align*}
		where \(\delta_z\) is the Dirac-mass of mass \(1\) concentrated at \(z\in\mathbb R^N\).
		
\noi Moreover, we can construct a smooth cut-off function \(\psi_{\e,i}\) centered at \(z_i\) such that
		\[0\leq \psi_{\e,i}(x)\leq 1,\ \psi_{\e,i}(x)=1\;\mbox{in}\; B\left(z_i,\e\right),  \psi_{\e,i}(x)=0\;\mbox{in}\; \mathbb R^N\setminus B (z_i,2\e),\ |\nabla \psi_{\e,i}|\leq \frac{2}{\e},\ |\De \psi_{\e,i}|\leq \frac{2}{\e^2}\]
		for any \(\e>0\) small.
		Since $\{\psi_{\epsilon,i} u_k\}$ is a bounded sequence in $D^{2,p}(\mathbb{R}^N)$, so we have
		\begin{align}\label{n8}
			\lim_{ k\rightarrow \infty}\left((a+b\|u_k\|^{p\theta})\right.\int_{\mathbb R^N}&|\Delta u_k|^{p-2}\Delta u_k \De(\psi_{\epsilon,i} u_k) - \la \int_{\mathbb R^N}f(x)|u_k|^r\psi_{\epsilon, i}\notag\\
   &\quad- \left.\alpha\int_{\mathbb R^N}\int_{\mathbb R^N} \frac{|u_k(x)|^{p^*_\mu}|u_k(y)|^{p^*_\mu}\psi_{\epsilon,i}}{|x-y|^{\mu}}\right) =0.
		\end{align}
		One can easily see that
		\begin{align*}
			\int_{\mathbb R^N}|\Delta u_k|^{p-2}\Delta u_k \De(\psi_{\epsilon,i} u_k)dx = \int_{\mathbb R^N}|\Delta u_k|^{p-2}\Delta u_k \left(\De \psi_{\epsilon,i} u_k + 2\nabla \psi_{\epsilon,i}\cdot\nabla u_k + \psi_{\e,i} \Delta u_k\right)dx.
		\end{align*}
		Now consider
		\begin{align*}
			0\leq &\limsup_{k \rightarrow \infty}\left|\int_{\mathbb R^N}|\Delta u_k|^{p-2}\Delta u_k (\nabla \psi_{\epsilon,i} .\nabla u_k)dx\right|\\
			\leq & \limsup_{k \rightarrow \infty}\left|\int_{\mathbb R^N}|\Delta u_k|^{p-1}|\nabla \psi_{\epsilon,i}||\nabla u_k|dx\right|\notag\\
			\leq & \limsup_{k \rightarrow \infty}\left(\int_{\mathbb R^N}|\Delta u_k|^p dx\right)^{\frac{p-1}{p}}\left(\int_{\mathbb R^N}|\nabla \psi_{\epsilon,i}|^p|\nabla u_k|^p dx\right)^{\frac 1p}\notag\\
			\leq & C\left(\int_{B(z_i, 2\epsilon)}|\nabla \psi_{\epsilon,i}|^p|\nabla u|^p dx\right)^{\frac 1p}\notag\\
			\leq & C \left[\left(\int_{B(z_i, 2\epsilon)}|\nabla \psi_{\epsilon,i}|^{N}dx\right)^{\frac pN}\left(\int_{B(z_i, 2\epsilon)}|\nabla u|^{\frac{Np}{N-p}} dx\right)^{\frac{N-p}{N}}\right]^{\frac{1}{p}}\notag\\
			\leq & C \left(\int_{B(z_i, 2\epsilon)}|\nabla \psi_{\epsilon,i}|^{N}dx\right)^{\frac 1N}\left(\int_{B(z_i, 2\epsilon)}|\nabla u|^{\frac{Np}{N-p}} dx\right)^{\frac{N-p}{Np}}\longrightarrow 0 \;\text{as}\; \epsilon \longrightarrow 0.
		\end{align*}
		Also using the same idea as above, we obtain
		\begin{align*}
			0\leq &\limsup_{k \rightarrow \infty}\left|\int_{\Om}|\Delta u_k|^{p-2}\Delta u_k (u_k\De \psi_{\epsilon,i} )dx\right| \leq \limsup_{k \rightarrow \infty}\left(\int_{\mathbb R^N}|\Delta u_k|^p dx\right)^{\frac{p-1}{p}}\left(\int_{\mathbb R^N}|\nabla \psi_{\epsilon,i}|^p|u_k|^p dx\right)^{\frac 1p}\notag\\
			\leq & C\left(\int_{B(z_i, 2\epsilon)}|\nabla \psi_{\epsilon,i}|^p|u|^p dx\right)^{\frac 1p}\notag\\
			\leq & C \left[\left(\int_{B(z_i, 2\epsilon)}|\nabla \psi_{\epsilon,i}|^{\frac N2}dx\right)^{\frac{2p}{N}}\left(\int_{B(z_i, 2\epsilon)}| u|^{p^*} dx\right)^{\frac{p}{p^*}}\right]^{\frac{1}{p}}\notag\\
			\leq & C\left(\int_{B(z_i, 2\epsilon)}|\nabla \psi_{\epsilon,i}|^{\frac N2}dx\right)^{\frac{2}{N}}\left(\int_{B(z_i, 2\epsilon)}|u|^{p^*} dx\right)^{\frac{1}{p^*}}\longrightarrow 0 \;\text{as}\; \epsilon \longrightarrow 0.
		\end{align*}
		Notice that
		\begin{align*}
			\left|\int_{\mathbb R^N} f(x)|u_k|^{r} \psi_{\e,i} dx \right| \leq& \int_{B_{\e}(z_i)} |f(x)||u_k|^{r} dx \leq \|f\|_{\frac{p^*}{p^*-r}}\left(\int_{B_{\e}(z_i)}|u_k|^{p^*} dx\right)^{\frac{r}{p^*}}\\
			\rightarrow & \|f\|_{\frac{p^*}{p^*- r}}\left(\int_{B_{\e}(z_i)}|u|^{p^*} dx\right)^{\frac{r}{p^*}}\longrightarrow 0 \;\text{as}\; \epsilon \longrightarrow 0.
		\end{align*}
		Thus
		\begin{align}\label{n12}
			\lim_{\e\ra 0} \lim_{k\ra \infty}\int_{\mathbb R^N} f(x)|u_k|^{r} \psi_{\e,i} dx =0.
		\end{align}
		Now, combining \eqref{n8}-\eqref{n12}, we deduce
		\begin{align*}
			0&= \lim_{\e\ra 0}\lim_{k\ra \infty} \langle \mc E^{\prime}_{\al,\la}(u_k), \psi_{\e,i} u_k \rangle\\
			&= \lim_{\e\ra 0}\lim_{k\ra\infty}\left\{(a+b \|u_k\|^{p\theta}) \int_{\mathbb R^N} |\Delta u_k|^{p-2}\Delta u_k \De (\psi_{\epsilon,i} u_k)  - \la \int_{\mathbb R^N} f(x) |u_k|^{r} \psi_{\e,i}dx \right.\\
			&\quad  \left.- \al\int_{\mathbb R^N}\int_{\mathbb R^N} \frac{|u_k(x)|^{p^*_\mu}|u_k(y)|^{p^*_\mu}\psi_{\epsilon,i}}{|x-y|^{\mu}}dxdy \right\}\\
			&\geq \lim_{\e\ra 0}\lim_{k\ra\infty}\left\{(a+b \|u_k\|^{p\theta}) \int_{\mathbb R^N}|\Delta u_k|^{p} \psi_{\epsilon,i}  - \al\int_{\mathbb R^N}\int_{\mathbb R^N} \frac{|u_k(x)|^{p^*_\mu}|u_k(y)|^{p^*_\mu}\psi_{\epsilon,i}}{|x-y|^{\mu}} dxdy\right\}\\
			& \geq \lim_{\e\ra 0}\lim_{k\ra\infty} \left\{ a \int_{\mathbb R^N} |\Delta u_k|^{p} \psi_{\e,i} dx - \al \int_{\mathbb R^N}\int_{\mathbb R^N} \frac{|u_k(x)|^{p^*_\mu}|u_k(y)|^{p^*_\mu}\psi_{\epsilon,i}}{|x-y|^{\mu}} dxdy \right\}\\
			& \geq \lim_{\e\ra 0}  \left\{a \int_{\mathbb R^N} \psi_{\e,i} d \omega - \al \int_{\mathbb R^N} \psi_{\e,i} d\nu \right\}\\
			&\geq a \omega_i- \al  \nu_i.
		\end{align*}
		Therefore, $a \omega_i\leq \al \nu_i$. Together this with the fact that $S_{H,L} \nu_{i}^{\frac{p}{2 p_{\mu}^*}} \leq\omega_i$, we obtain 
		\[ \mbox{either}\quad \omega_{i}\geq \left( a \alpha^{-1}  S_{H,L}^{\frac{2N-\mu}{N-2p}}\right)^{\frac{N-2p}{N-\mu +2p}}\quad \mbox{or}\quad \omega_i=0. \]
		Now, we claim that the first case can not occur. Suppose not, then there exists $i_0\in J$ such that $\omega_{i_0}\geq \left( a \alpha^{-1}  S_{H,L}^{\frac{2N-\mu}{N-2p}}\right)^{\frac{N-2p}{N-\mu +2p}}$.
		Equation \eqref{e3}, the H\"{o}lder inequality, the Sobolev inequality and the Young inequality imply that
		\begin{align*}
			\la \int_{\mathbb R^N} f(x) |u|^{r} dx &\leq \la \|f\|_{\frac{p^*}{p^*-r}} S^{-\frac{r}{p}} \|u\|^{r}= \left(\left[\left(\frac{1}{p}- \frac{1}{2p_{\mu}^*}\right)\frac{a}{r}\left(\frac{1}{r}- \frac{1}{2p_{\mu}^*}\right)^{-1}\right]^{\frac{r}{p}}\|u\|^{r}\right)\notag\\
			& \quad\quad \left(\left[\left(\frac{1}{p}- \frac{1}{2p_{\mu}^*}\right) \frac{a}{r}\left(\frac{1}{r}- \frac{1}{2p_{\mu}^*}\right)^{-1}\right]^{\frac{-r}{p}} \la \|f\|_{\frac{p^*}{p^*-r}} S^{-\frac{r}{p}}\right)\notag\\
			&\leq \left(\frac{1}{p}- \frac{1}{2p_{\mu}^*}\right) \frac{a}{p}\left(\frac{1}{r}- \frac{1}{2p_{\mu}^*}\right)^{-1} \|u\|^p\notag\\
			& \quad \quad + \frac{p-r}{p}\left[\left(\frac{1}{r}- \frac{1}{2p_{\mu}^*}\right) \frac{r}{aS}\left(\frac{1}{p}- \frac{1}{2p_{\mu}^*}\right)^{-1}\right]^{\frac{r}{p-r}} {\la}^{\frac{p}{p-r}} \|f\|_{\frac{p^*}{p^*-r}}^{\frac{p}{p-r}}
		\end{align*}
		Using this fact, we have
		{\small \begin{align}\label{e6}
				0&>c= \lim_{k\ra \infty}\left(\mc E_{\al, \la}(u_k) - \frac{1}{2 p_{\mu}^*} \langle \mc E^{\prime}_{\al,\la}(u_k), u_k \rangle \right)\notag\\
				&= \lim_{k\ra\infty}\left\{\left(\frac{1}{p}- \frac{1}{2p_{\mu}^*}\right) a \|u_k\|^{p} + \left(\frac{1}{p\theta}- \frac{1}{2p_{\mu}^*}\right) b \|u_k\|^{p\theta} - \left(\frac{1}{r}- \frac{1}{2p_{\mu}^*}\right) \la \int_{\mathbb R^N} f(x) |u_k|^{r} dx \right\}\notag\\
				& \geq \left(\frac{1}{p}- \frac{1}{2p_{\mu}^*}\right) a \left(\|u\|^{p}+ \sum_{i\in J} \omega_i \right)- \left(\frac{1}{r}- \frac{1}{2p_{\mu}^*}\right) \la \int_{\mathbb R^N} f(x) |u_k|^{r} dx\notag \\
				&\geq \left(\frac{1}{p}- \frac{1}{2p_{\mu}^*}\right) a \omega_{i_0} - \frac{p-r}{p}\left(\frac{1}{r}- \frac{1}{2p_{\mu}^*}\right)\left[\left(\frac{1}{r}- \frac{1}{2p_{\mu}^*}\right) \frac{r}{aS}\left(\frac{1}{p}- \frac{1}{2p_{\mu}^*}\right)^{-1}\right]^{\frac{r}{p-r}} {\la}^{\frac{p}{p-r}} \|f\|_{\frac{p^*}{p^*-r}}^{\frac{p}{p-r}}\notag\\
				&\geq \left(\frac{1}{p}- \frac{1}{2p_{\mu}^*}\right)\left(a  S_{H,L}\alpha^{-\frac{N-2p}{2N-\mu}}\right)^{\frac{2N-\mu}{N-\mu +2p}} - \frac{p-r}{p}\left(\frac{1}{r}- \frac{1}{2p_{\mu}^*}\right)\left[\left(\frac{1}{r}- \frac{1}{2p_{\mu}^*}\right) \frac{r}{aS}\left(\frac{1}{p}- \frac{1}{2p_{\mu}^*}\right)^{-1}\right]^{\frac{r}{p-r}}\notag\\&\quad\quad\times {\la}^{\frac{p}{p-r}} \|f\|_{\frac{p^*}{p^*-r}}^{\frac{p}{p-r}}.
		\end{align}}
       
        \noi Thus, for any $\al>0$, we choose $\la_1>0$ so small that for every $\la\in(0,\la_1)$, the right hand side of \eqref{e6} is greater than zero, which gives a contradiction.
		
		\noi Similarly, if for any $\la>0$, we choose $\al_1>0$ so small that for every $\al\in(0,\al_1)$, the right hand side of \eqref{e6} is greater than zero, which gives a required contradiction. Consequently, $\omega_{i}=0$ for all $i\in J$.
		
	\noi	To obtain the possible concentration of mass at infinity, we can define a cut-off function \(\psi_{R}\in C^{\infty}(\mathbb R^N)\) such that \(\psi_{R}(x)=0\) on \(|x|<R\), \(\psi_{R}(x)=1\) on \(|x|>R+1\), \(|\nabla \psi_{R}|\leq \frac{2}{R}\) and \(|\De \psi_{R}|\leq \frac{2}{R^2}\).
		
	\noi	Now applying the Hardy-Littlewood-Sobolev, H\"{o}lder's inequalitity, we have
		\begin{align*}
			\nu_{\infty}&= \lim_{R\ra \infty} \lim_{k\ra \infty} \int_{\mathbb R^N}\left(\int_{\mathbb R^N}\frac{|u_k(y)|^{p^*_{\mu}}}{|x-y|^{\mu}} dy\right) |u_k(x)|^{p^*_{\mu}} \psi_{R}(x) dx\\
			&\leq C(N,\mu) \lim_{R\ra \infty} \lim_{k\ra \infty} |u_k|^{p^{*}_{\mu}}_{p^*} \left(\int_{\mathbb R^N}|u_k(x)|^{p^*} \psi_{R} dx\right)^{\frac{p^{*}_{\mu}}{p^{*}}}\\
			&\leq K \zeta_{\infty}^{\frac{p_{\mu}^{*}}{p^*}}.
		\end{align*}
		Using the relation \(\langle \mc E_{\al,\la}^{\prime}(u_k), u_k \psi_{R}\rangle \ra 0\), we obtain
		\begin{align}\label{s1.1}
			0&= \lim_{R\ra\infty}\lim_{k\ra \infty} \langle \mc E^{\prime}_{\al,\la}(u_k), \psi_{R} u_k \rangle\notag\\
			&= \lim_{R\ra\infty}\lim_{k\ra\infty}\left\{(a+b \|u_k\|^{p\theta}) \int_{\mathbb R^N} |\Delta u_k|^{p-2} \Delta u_k \De (\psi_R u_k)  - \la \int_{\mathbb R^N} f(x) |u_k|^{r} \psi_{R}dx \right.\notag\\
			&\quad  \left.- \al\int_{\mathbb R^N}\int_{\mathbb R^N} \frac{|u_k(x)|^{p^*_\mu}|u_k(y)|^{p^*_\mu}\psi_{R}}{|x-y|^{\mu}} \right\}\notag\\
			& \geq \lim_{R\ra\infty}\lim_{k\ra\infty} \left\{ a \int_{\mathbb R^N} |\Delta u_k|^{p} \psi_{R} dx - \al \int_{\mb R^N}\int_{\mb R^N} \frac{|u_k(x)|^{p^*_\mu}|u_k(y)|^{p^*_\mu}\psi_R}{|x-y|^{\al}} dxdy \right\}\notag\\
			& \geq \lim_{R\ra\infty}  \left\{a \int_{\mathbb R^N} \psi_{R} d \omega - \al \int_{\mathbb R^N} \psi_{R} d\nu \right\}\notag\\
            &\geq a \omega_{\infty} - \al  \nu_{\infty}\\
			&\geq a \omega_{\infty} - \al K \zeta_{\infty}^{\frac{p_{\mu}^*}{p^*}}.\notag
		\end{align}
		Therefore, $a \omega_{\infty} \leq \al K \zeta_{\infty}^{\frac{p_{\mu}^*}{p^*}}$. Combining this with Lemma \ref{1e0}, we obtain 
		\[ \mbox{either}\quad \omega_{\infty}\geq \left( a \alpha^{-1} K^{-1}  S^{\frac{p^*_\mu}{p}}\right)^\frac{p}{p^*_\mu-p}\quad \mbox{or}\quad \omega_{\infty}=0.  \]
	\noi	Now,
		\begin{align}\label{e7}
			0>c&\geq \left(\frac{1}{2p}- \frac{1}{4p_{\mu}^*}\right)\left(aS\right)^{\frac{p^*_\mu}{p^{*}_{\mu}-p}}(K\al)^{-\frac{p}{p^{*}_{\mu}-p}} \notag\\
   & \quad - \frac{p-r}{p}\left(\frac{1}{r}- \frac{1}{2p_{\mu}^*}\right)\left[\left(\frac{1}{r}- \frac{1}{2p_{\mu}^*}\right) \frac{r}{aS}\left(\frac{1}{p}- \frac{1}{2p_{\mu}^*}\right)^{-1}\right]^{\frac{r}{p-r}} {\la}^{\frac{p}{p-r}} \|f\|_{\frac{p^*}{p^*-r}}^{\frac{p}{p-r}}.
		\end{align}
		Thus, for any $\al>0$, we choose $\la_2>0$ so small that for every $\la\in(0,\la_2)$, the right hand side of \eqref{e7} is greater than zero, which gives a contradiction.	
\noi		Similarly, if for any $\la>0$, we choose $\al_2>0$ so small that for every $\al\in(0,\al_2)$, the right hand side of \eqref{e7} is greater than zero, which gives a required contradiction. Consequently, $\omega_{\infty}=0$.

		\noi From the above arguments, Take $\overline{\Lambda}=\min\{\la_1, \la_2\}$ and $\underline{\Lambda}=\min\{\al_1, \al_2\}$.
		
		\noi Then for any $c<0$, $\al>0$, we have \(\omega_i=0\) for all $i\in J$ and $\omega_{\infty}=0$ for all \(\la\in (0, \overline{\Lambda})\).
		
	\noi	Similarly, for any $c<0$, $\la>0$, we have \(\omega_i=0\) for all $i\in J$ and $\omega_{\infty}=0$ for all \(\al\in (0, \underline{\Lambda})\).
		
	\noi	Hence,
		\begin{align*}
			\lim_{k\ra \infty}\int_{\mathbb R^N}\int_{\mathbb R^N} \frac{|u_k(x)|^{p^*_\mu}|u_k(y)|^{p^*_\mu}}{|x-y|^{\mu}} dx dy = \int_{\mathbb R^N}\int_{\mathbb R^N} \frac{|u(x)|^{p^*_\mu}|u(y)|^{p^*_\mu}}{|x-y|^{\mu}} dx dy,
		\end{align*}
		
		\begin{align*}
			\lim_{k\ra \infty}\int_{\mathbb R^N} f(x) (|u_k(x)|^{r} - |u(x)|^{r}) dx \leq \|f\|_{\frac{p^*}{p^*-r}} \||u_k(x)|^{r} - |u(x)|^{r}\|_{\frac{p^*}{r}} = 0.
		\end{align*}
Since $(\|u_k\|)_{k}$ is bounded and $\mc E^{\prime}_{\al,\la}(u)=0$, the weak lower semicontinuity of the norm and the Br{\'e}zis-Lieb Lemma yield as $(k\ra \infty)$
		{\small\begin{align*}
			o(1)&= \langle \mc E_{\al,\la}^{\prime}(u_k),u_k \rangle = a \|u_k\|^{p} + b \|u_k\|^{p\theta} - \la \int_{\mathbb R^N} f(x) |u_k|^{r} dx - \al\int_{\mathbb R^N}\int_{\mathbb R^N} \frac{|u_k(x)|^{p^*_\mu}|u_k(y)|^{p^*_\mu}}{|x-y|^{\mu}} dx dy  \\
			&\geq a(\|u_k - u\|^p) + a \|u\|^p + b \|u\|^{p\theta} - \la \int_{\mathbb R^N} f(x) |u|^{r} dx - \al\int_{\mathbb R^N}\int_{\mathbb R^N} \frac{|u(x)|^{p^*_\mu}|u(y)|^{p^*_\mu}}{|x-y|^{\mu}} dx dy \\
			&= a \|u_k -u\|^p + o(1).
		\end{align*}}
		
        \noi Thus $\{u_k\}$ converges strongly to $u$ in $D^{2,p}(\mathbb R^N)$. This completes the proof of the Lemma.\qed
\end{proof}
\begin{lemma}\label{ce}
Let $r=p$ and $p\theta < 2 p_{\mu}^{*}$. Suppose that \(\{u_k\}\) is a $(PS)_c$ sequence for $\mc E_{\al,\la}$ in $D^{2,p}(\mathbb R^N)$, with $$c<c^*:=min\{c_1,c_2\},$$
where $c_1:=\left(\frac{1}{ p}- \frac{1}{2p^*_\mu}\right)\left(a  S_{H,L}
\alpha^{-\frac{N-2p}{2N-\mu}}
\right)^{\frac{2N-\mu}{N-\mu +2p}}$ and  $c_2:=\left(\frac{1}{p} -\frac{1}{2p_{\mu}^*}\right)\left( aS\right)^\frac{p^*_\mu}{p^*_\mu-p} (\al K)^{-\frac{p}{p^*_\mu-p}}$.\\
\noi Then for all \(\la\in\left(0,aS\|f\|_{\frac{p^*}{p^*-p}}^{-1}\right)\), $\{u_k\}$ satisfies the $(PS)_c$  condition.
\end{lemma}
\begin{proof}
For \(u\in D^{2,p}(\mathbb R^N)\) and $r=p$, the H\"{o}lder inequality and Sobolev inequality  imply that
\begin{align*}
		\int_{\mb R^N} f(x)|u|^{p} dx \leq S^{-1} \|f\|_{\frac{p^*}{p^*-p}} \|u\|^p.
	\end{align*}
	Let $\{u_k\}$ be a $(PS)_c$	for $\mc E_{\al,\la}$ for $c<c^*$. Then $\{u_k\}$ is bounded from Lemma \ref{l1}. Now using the last estimate, for  all \(\la\in\left(0,aS\|f\|_{\frac{p^*}{p^*-p}}^{-1}\right)\), arguing similarly as in Lemma \ref{p11}, in substitute of \eqref{e6}, we obtain
\begin{align*}
		c^*> c &= \lim_{k\ra \infty} \left(\mc E_{\al,\la}(u_k) - \frac{1}{p\theta }\left\langle \mc E^{\prime}_{\al,\la}(u_k), u_k\right\rangle\right)\\
        &\geq \left(\frac{1}{p}- \frac{1}{p\theta }\right)\left[a||u||^p+a\displaystyle\sum_{i\in J} \omega_i-\la S^{-1} \|f\|_{\frac{p^{*}}{p^*-p}}\|u\|^p \right]\\&\quad +\alpha\left(\frac{1}{p\theta }- \frac{1}{ 2p^*_\mu}\right)\left[\int_{\mathbb R^N}\int_{\mathbb R^N} \frac{|u(x)|^{p^*_\mu}|u(y)|^{p^*_\mu}}{|x-y|^{\mu}} dx dy+\displaystyle\sum_{i\in J} \nu_i\right] \\
&\geq \left(\frac{1}{p}- \frac{1}{p\theta }\right)a\omega_{i_{0}}+ \left(\frac{1}{p}- \frac{1}{p\theta }\right) \left(a- \la S^{-1} \|f\|_{\frac{p^{*}}{p^*-p}}\right)\|u\|^p+\alpha\left(\frac{1}{p\theta }- \frac{1}{ 2p^*_\mu}\right)\nu_{i_{0}}\\
&\geq\left(\frac{1}{p}- \frac{1}{p\theta }\right)\left(a  S_{H,L}
\alpha^{-\frac{N-2p}{2N-\mu}}
\right)^{\frac{2N-\mu}{N-\mu +2p}} +\left(\frac{1}{p\theta }- \frac{1}{ 2p^*_\mu}\right)\left(a  S_{H,L}
\alpha^{-\frac{N-2p}{2N-\mu}}
\right)^{\frac{2N-\mu}{N-\mu +2p}}\\
&=\left(\frac{1}{ p}- \frac{1}{ 2p^*_\mu}\right)\left(a  S_{H,L}
\alpha^{-\frac{N-2p}{2N-\mu}}
\right)^{\frac{2N-\mu}{N-\mu +2p}}:=c_1.,
	\end{align*}
    \noi Following the same argument as in Lemma \ref{p11} for concentration of mass at infinity and using equation \eqref{s1.1}, we obtain
\begin{align*}
   c &= \lim_{n\ra \infty} \left\{ \mc E_{\al,\la}(u_k) - \frac{1}{p\theta}\left\langle \mc E^{\prime}_{\al,\la}(u_k), u_k \right\rangle\right\}\nonumber\\
          &\geq  \left(\frac1p -\frac{1}{p\theta}\right) \left[a
   -\la S^{-1}||f||_{\frac{p^*}{p^*-p}}\right]||u||^p + a\left(\frac{1}{p} -\frac{1} 
    {p\theta}\right) \ \omega_{\infty}+ \al\left(\frac{1}{p\theta} -\frac{1}{2p_{\mu}^*}\right) \nu_{\infty}\nonumber\\
    &\geq a\left(\frac{1}{p} -\frac{1} 
    {p\theta}\right) \ \omega_{\infty}+ \left(\frac{1}{p\theta} -\frac{1}{2p_{\mu}^*}\right) a\omega_{\infty}\nonumber\\
          &\geq \left(\frac{1}{p} -\frac{1}{2p_{\mu}^*}\right)\left( aS\right)^\frac{p^*_\mu}{p^*_\mu-p} (\al K)^{-\frac{p}{p^*_\mu-p}}:= c_2.
\end{align*}
	which is absurd since $c<c^*:=min\{c_1, c
     _2\}$. Now the rest of the proof follows in similar manner as in the proof of Lemma \ref{l2}.
\end{proof}
\begin{lemma}\label{cc}
Let $p\theta \leq r<p^*$ and $p\theta <2p_{\mu}^*$. Suppose that $\{u_k\}$ be a $(PS)_c$ sequence for $\mc E_{\al,\la}$ in $D^{2,p}(\mathbb R^N)$ with $c<c^{*}
,$ where $c^*$ is defined same as in Lemma \ref{ce}. Then $\{u_k\}$ satisfies $(PS)_c$ condition.
\end{lemma}
\begin{proof}
	Let $\{u_k\}$ be a $(PS)_c$	for $\mc E_{\al,\la}$ for $c<c^{*}$.  Then by Lemma \ref{l1}, we have $\{u_k\}$ is bounded. Now  following the similar arguments as in Lemma \ref{p11},  we get
	\begin{align*}
		c^*> c &= \lim_{k\ra \infty} \left(\mc E_{\al,\la}(u_k) - \frac{1}{r}\left\langle \mc E^{\prime}_{\al,\la}(u_k), u_k\right\rangle\right)\\
        &= \lim_{k\ra \infty}\left\{a\left(\frac{1}{p}-\frac{1}{r}\right)||u_k||^p+b\left(\frac{1}{p\theta }-\frac{1}{r}\right)||u_k||^{p\theta }+\al\left(\frac{1}{r}-\frac{1}{2p^*_\mu}\right)||u_k||^{2p^*_\mu}_*\right\} \\
		& \geq \left(\frac{1}{p}-\frac{1}{r}\right) a \omega_{i_0}+\left(\frac{1}{r}-\frac{1}{2p^*_\mu}\right)\alpha\nu_{i_{0}}
		\geq\left(\frac{1}{p}-\frac{1}{2p^*_\mu}\right)\left(a  S_{H,L}\alpha^{-\frac{N-2p}{2N-\mu}}\right)^{\frac{2N-\mu}{N-\mu +2p}}:= c_1.
	\end{align*}
    \noi   Following the same  as in Lemma \ref{p11}, for mass at infinity and using \eqref{s1.1}, we obtain
    \begin{align*}
        c^{*}> c &= \lim_{k\ra \infty} \left(\mc E_{\al,\la}(u_k) - \frac{1}{r}\left\langle \mc E^{\prime}_{\al,\la}(u_k), u_k\right\rangle\right)\\
        & \geq \left(\frac{1}{p}-\frac{1}{r}\right) a \omega_\infty+\left(\frac{1}{r}-\frac{1}{2p^*_\mu}\right)\alpha\nu_\infty\nonumber\\
          &\geq \left(\frac{1}{p} -\frac{1}{2p_{\mu}^*}\right)\left( aS\right)^\frac{p^*_\mu}{p^*_\mu-p} (\al K)^{-\frac{p}{p^*_\mu-p}}:= c_2.
\end{align*}
	which is contradiction since $c<c^*:=min \{c_1, c_2\}$. The rest of the proof follows as in the proof of Lemma \ref{l2}.
\end{proof}
\subsection{ Case 2: \texorpdfstring{$p\theta  \geq 2 p_{\mu}^{*}$}{ }}
\noi \noi In this subsection, we prove  the strong convergence of the $(PS)_c$ sequence in both the cases non-degenerate and degenerate.\\ 
\noi To start with the energy functional $\mc E_{\la}$ corresponding to the problem $(\mc P_{ \la})$ when $\al=1$
\[\mc E_{\la}(u)= \frac{a}{p} \|u\|^{p}+ \frac{b}{p\theta} \|u\|^{p\theta} - \frac{1}{2 p_{\mu}^*}\int_{\mathbb R^N}\int_{\mathbb R^N}\frac{|u(x)|^{p_{\mu}^*}|u(y)|^{p_{\mu}^*}}{|x-y|^{\mu}}dxdy - \frac{\la}{r}\int_{\mathbb R^N}f(x)|u|^r dx.\]
\begin{lemma}\label{L6}
 Let  $p\geq 2$, $f$ satisfies $(f_1)-(f_2)$ and $p\theta \geq 2p_{\mu}^{*}$ . Then $\mc E_{\la}$ satisfies the $(PS)_c$ condition in $D^{2,p}(\mathbb R^N)$ for all $\la>0$, in the following cases:
 \begin{enumerate}
 \item  $p<r<p^*$. If either $p\theta = 2p_{\mu}^{*}$, $a>0$ and $b>  2^{p} S_{H,L}^{-\frac{2p_{\mu}^{*}}{p}}$ or $p\theta> 2 p_{\mu}^{*}$, $a>0$ and $b>b^{1}$, where $b^1$ is defined in \eqref{eb}. 
 \item  $1\leq r<p^*$, $a=0$, $b>2^{p} S_{H,L}^{-\frac{2p_{\mu}^{*}}{p}}$.
 \end{enumerate}
\end{lemma}
\begin{proof}
For all \(\la>0\), suppose \(\{u_{k}\}\) be a Palais-Smale sequence of \(\mc E_{\la}\) in \(D^{2,p}(\mathbb R^N)\) at any level \(c<0\). Then by Lemma \ref{l1}, $\{u_k\}$ is a bounded in $D^{2,p}(\mathbb R^N)$.  Therefore as $k\ra \infty$ we can assume up to a subsequence still denoted by $\{u_k\}$ such that \(u_k \rightharpoonup u\) weakly in $D^{2,p}(\mathbb R^N)$, $u_k \ra u$ a.e in $\mathbb R^N$, \(u_k \ra u\)  strongly in $L^{q}_{loc}(\mb R^N)$ for  $1\leq q< p^*$ and $|u_k|^{p^*-2} u_k \rightharpoonup |u|^{p^*-2} u$ weakly in $L^{\frac{p^*}{p^*-1}}(\mb R^N)$. 
\noi We claim that 
\begin{align}\label{11}
\lim_{k\ra\infty}\int_{\mb R^N} f(x) |u_k|^r dx= \int_{\mb R^N} f(x) |u|^r dx.
\end{align}
As $f\in L^{\frac{p^{*}}{p^*-r}}(\mb R^N)$, for any $\e>0$, there exist $R_{\e}>0$ such that
\[\left(\int_{\mb R^N\setminus B_{R_{\e}}(0)} |f(x)|^{\frac{p^{*}}{p^*-r}} dx \right)^{\frac{p^{*}}{p^*-r}} <\e.\]
\noi Then H\"{o}lder inequality and above inequality yield
\begin{align}
\left|\int_{\mb R^N\setminus B_{R_{\e}}(0)} f(x) (|u_k|^r -|u|^r) dx\right|& \leq \left(\int_{\mb R^N\setminus B_{R_{\e}}(0)} |f(x)|^{\frac{p^{*}}{p^*-r}} dx \right)^{\frac{p^{*}}{p^*-r}}\left(\|u_k\|_{L^{p^*}(\mathbb R^N)}+ \|u\|_{L^{p^*}(\mathbb R^N)}\right)\notag\\
&< C_\e.\label{33}
\end{align} 
Now, for any non-empty measurable subset $\Omega\subset B_{R_{\e}}$ and boundedness of $\{u_k\}$ give
\[\left|\int_{\mb R^N\setminus B_{R_{\e}}(0)} f(x) (|u_k|^r -|u|^r) dx\right| \leq C\left(\int_{\mb R^N\setminus B_{R_{\e}}(0)} |f(x)|^{\frac{p^{*}}{p^*-r}} dx \right)^{\frac{p^{*}}{p^*-r}}.\]
\noi This implies the sequence $\{f(x)(|u_k|^r-|u|^r)\}$ is equi-integrable  in $B_{R_{\e}}(0)$. Hence Vitali convergence Theorem implies 
\begin{align}\label{22}
\lim_{k\ra\infty}\int_{B_{R_{\e}}(0)} f(x) |u_k|^r dx= \int_{B_{R_{\e}}(0)} f(x) |u|^r dx.
\end{align}
Combining \eqref{33} and \eqref{22}, conclude the claim \eqref{11}.
Moreover, one can easily deduce that
\begin{align}\label{eqp1}
\lim_{k\ra \infty}\int_{\mb R^N} f(x)[|u_k|^{r-2} u_k - |u|^{r-2} u](u_k- u) dx =0.
\end{align}
Then, we may assume that $\displaystyle\lim_{k\ra \infty} \|u_k-u\|=l$
\[\|u_k- u\|_{*}^{2 p_{\mu}^{*}}=\|u_k\|_{*}^{2 p_{\mu}^{*}}- \|u\|_{*}^{2 p_{\mu}^{*}}.\]
Also, the definition of weak convergence in $D^{2,p}(\mathbb R^N)$,
\[\lim_{k\ra \infty}\int_{\mathbb R^N}|\Delta u|^{p-2} \Delta u \De(u_k-u) dx=0.\]
Since $\{u_k\}$ is a $(PS)_c$ sequence, by the boundedness of $\{u_k\}$, \eqref{eqp1}, we have
\begin{align*}
o(1)&= \langle \mc E_{\al,\la}^{\prime}(u_k)-  \mc E_{\al,\la}^{\prime}(u), u_k-u\rangle  \notag\\
&= (a+b\|u_k\|^{p(\theta-1)})\int_{\mathbb R^N}|\Delta u_k|^{p-2} \Delta u_k \De(u_k-u) dx - \la \int_{\mb R^N} f(x)[|u_k|^{r-2} u_k - |u|^{r-2} u]\notag\\ 
&\quad  \times(u_k- u) dx-(a+b\|u\|^{p(\theta-1)})\int_{\mathbb R^N}|\Delta u|^{p-2} \Delta u \De(u_k-u) dx\notag\\
&\quad  -  \int_{\mb R^N}\int_{\mb R^N} \left[\frac{|u_k(y)|^{p_{\mu}^*}|u_k(x)|^{p_{\mu}^*-2} u_k}{|x-y|^{\mu}}- \frac{|u(y)|^{p_{\mu}^*}|u(x)|^{p_{\mu}^*-2} u}{|x-y|^{\mu}}\right](u_k-u) dxdy\notag \\
&=(a+b\|u_k\|^{p(\theta-1)}) 
\int_{\mathbb R^N}|\Delta u_k|^{p-2} \Delta u_k \De(u_k-u) dx \notag\\
&\quad\quad-  {\int\int}_{\mb R^{2N}} \frac{|(u_k-u)(y)|^{p_{\mu}^*}|(u_k-u)(x)|^{p_{\mu}^*}}{|x-y|^{\mu}} dxdy + o(1),
\end{align*}

\noi Now using the following inequality, for any $p\geq 2$, 
\[(|a|^{p-2}a- |b|^{p-2}b)(a-b)\geq \frac{1}{2^p}|a-b|^p\; \mbox{for all}\; a, b\in \mathbb R,\]
together with the definition of $S_{H,L}$ yield 
\[\left(a+ b\|u_k-u\|^{(p-1)\theta}\right)\frac{1}{2^p}\|u_k-u\|^{p} \leq  S_{H,L}^{-\frac{2p_{\mu}^{*}}{p}} \|u_k-u\|^{2p_{\mu}^{*}}.\]
Taking limit $k\ra \infty$, we have
\[a l^p+ b(l^p+ \|u\|^p)^{\theta-1} l^p\leq  2^p S_{H,L}^{-\frac{2p_{\mu}^{*}}{p}} l^{2p_{\mu}^*},\]
which imply
\begin{align}\label{pl}
a l^p+ b l^{p\theta} \leq  2^p S_{H,L}^{-\frac{2p_{\mu}^{*}}{p}} l^{2p_{\mu}^*}.
\end{align}
Case 1 : When $a>0$, $p\theta= 2 p_{\mu}^{*}$, and $ 2^p S_{H,L}^{-\frac{2p_{\mu}^{*}}{p}}<b$, from \eqref{pl}, one can easily deduce that $l=0$. Thus $u_k\ra u$ in $D^{2,p}(\mathbb R^N)$. 

\noi Case 2: When $a=0$, $p\theta= 2 p_{\mu}^{*}$, and $ 2^p S_{H,L}^{-\frac{2p_{\mu}^{*}}{p}}<b$,
 equation \eqref{pl} yield that $l=0$. Thus $u_k\ra u$ in $D^{2,p}(\mathbb R^N)$. 

\noi Case 3:  When $a>0$, $p\theta> 2 p_{\mu}^{*}$ and $b>b^1$.

\noi Applying Young's inequality in the right hand side of \eqref{pl}, we obtain
\begin{align*}
al^p+ b l^{p\theta} &\leq  a l^p 
+\frac{2 p_{\mu}^{*}-p}{p(\theta-1)}\left[\left(\frac{ap(\theta-1)}{(p\theta- 2p_{\mu}^{*})}\right)^{\frac{2p_{\mu}^{*}-p\theta }{2p_{\mu}^{*}-p}} \left(2^p S_{H,L}^{-\frac{2p_{\mu}^{*}}{p}}\right)^{\frac{p(\theta-1)}{2p_{\mu}^{*}-p}} l^{p\theta}\right]\notag\\
&= a l^p + b^{1} l^{p\theta},
\end{align*}
where $b^{1}$ is given in \eqref{eb}. Thus $(b- b^{*})l^{p\theta}\leq 0$. In view of \eqref{eb}, we deduce $l=0$. 
Hence, we obtain that $u_k \ra u$  strongly in $D^{2,p}(\mathbb R^N)$, as required.
\end{proof}
\section{Case: \texorpdfstring{$p\theta \geq2p_{\mu}^*$}{ }}
\noi In this section, first we show that the functional is coercive and bounded below and then we prove the Theorems \ref{th3} and \ref{th4}.

\begin{lemma}\label{41}
Let $p\theta \geq 2p_{\mu}^*$,  $\al=1$ and  $1<r<p^*$. Then show that $\mc E_{\la}$ is coercive and bounded below for all $a\geq0$ and $ b > 
	\begin{cases}
         &S^{-\frac{2p^*_\mu}{p}}_{H,L}\mbox{\;\; if} \ p\theta =p^*_{\mu} \\
		&0 \quad \quad\ \mbox{\;\; if}\ p\theta >p^*_{\mu}
	\end{cases}$.
\end{lemma}
\begin{proof}
Let $u\in D^{2,p}(\mathbb R^N)$. Then
equations \eqref{n3} and \eqref{e3} yield
\begin{align*}
\mc E_{\la}&(u)=\frac{a}{p} \|u\|^{p}+ \frac{b}{p\theta} \|u\|^{p\theta} - \frac{1}{2 p_{\mu}^*}\int_{\mathbb R^N}\int_{\mathbb R^N}\frac{|u(x)|^{p_{\mu}^*}|u(y)|^{p_{\mu}^*}}{|x-y|^{\mu}}dxdy - \frac{\la}{r}\int_{\mathbb R^N}f(x)|u|^r dx 
\end{align*}
\noi Case 1: $a\geq0$, $p<r<p^*$ and $p\theta > 2p_{\mu}^*$,
\begin{align*}
    \mc E_{\la}&(u)\geq \frac{a}{p}\|u\|^p 
+ \frac{b}{p\theta}\|u\|^{p\theta}-\frac{1}{2 p_{\mu}^*} S_{H,L}^{-\frac{2p_{\mu}^*}{p}}\|u\|^{2 p_{\mu}^*}-\frac{\la}{r} S^{-\frac rp}\|f\|_{\frac{p^*}{p^*-r}}\|u\|^r.
\end{align*}
\noi Case 2: $a=0$, $1<r<p^*$ and $p\theta = 2p_{\mu}^*$,
\[\mc E_{\la}(u)\geq \frac{1}{2p_{\mu}^*}\left(b- S_{H,L}^{-\frac{2p_{\mu}^*}{p}}\right)\|u\|^{2 p_{\mu}^*}-\frac{\la}{r} S^{-\frac rp}\|f\|_{\frac{p^*}{p^*-r}}\|u\|^r. \]
\noi Case 2:  $a>0$, $p<r<p^*$ and $p\theta = 2p_{\mu}^*$, 
\[\mc E_{\la}(u)\geq  \frac{a}{p}\|u\|^p  +\frac{1}{2p_{\mu}^*}\left(b- S_{H,L}^{-\frac{2p_{\mu}^*}{p}}\right)\|u\|^{2 p_{\mu}^*}-\frac{\la}{r} S^{-\frac rp}\|f\|_{\frac{p^*}{p^*-r}}\|u\|^r.\]

\noi  In all the cases, we  conclude that $\mc E_{\la}$ is bounded below and coercive.
\end{proof}
\noi So, we define $m := \displaystyle\inf_{u\in D^{2,p}(\mathbb R^N)}\mc E_{\la}(u)$,
which is well defined by Lemma \ref{41}.\\

\noi {\bf {Proof of Theorem \ref{th3}:}} We first show that the problem $(\mc P_{\la})$ has a non-trivial least energy solution. 

\noi We claim that there exists $\la^1>0$ such that $m<0$ for all $\la>\la^1$.

\noi Choose a function $\Phi \in D^{2,p}(\mathbb R^N)$ with $\|\Phi\|=1$ and $\int_{\mathbb R^N} f(x) |\Phi|^r dx>0$, which is possible due to $f\geq 0$ and $f\not\equiv 0$ in $\mathbb R^N$. Then
\begin{align*}
\mc E_{\la}(\Phi)&= \frac{a}{p} + \frac{b}{p\theta} -\frac{1}{2 p_{\mu}^{*}}\int_{\mb R^N}\int_{\mb R^N} \frac{|\Phi(x)|^{p_{\mu}^{*}}|\Phi(y)|^{p_{\mu}^{*}}}{|x-y|^{\mu}} dx dy-\frac{\la}{r} \int_{\mb R^N} f(x) |\Phi|^r dx\\
&\leq \frac{a}{p} + \frac{b}{p\theta}-\frac{\la}{r} \int_{\mb R^N} f(x) |\Phi|^r dx<0,
\end{align*}
for all $\la>\la^1$ with $\la^1= \frac{r\left(\frac{a}{p}+ \frac{b}{p\theta}\right)}{\int_{\mb R^N} f(x) |\Phi|^r dx}$.

\noi Hence, by Lemma \ref{L6} and [\cite{31}, Theorem 4.4], there exists $u_1\in D^{2,p}(\mathbb R^N)$ such that $\mc E_{\la}(u_1)=m$ and $\mc E^{\prime}_{\la}(u_1) =0$. Thus $u_1$ is a non-trivial solution of $(\mc P_{\la})$ with $\mc E_{\la}(u_1)<0$.

\noi Next, we show that $(\mc P_{\la})$ has a mountain pass solution.  For all $u\in D^{2,p}(\mb R^N)$, using \eqref{e3} and \eqref{n3}, we obtain
\[\mc E_{\la}(u)\geq \left[\frac{a}{p} + \frac{b}{p\theta}\|u\|^{p\theta-p}- \la \|f\|_{\frac{p^*}{p^*-r}}S^{-\frac{r}{p}}\|u\|^{r-p}- \frac{ S_{H,L}^{-\frac{2p_{\mu}^{*}}{p}}}{2 p_{\mu}^*}\|u\|^{2p_{\mu}^{*}-p}\right]\|u\|^p.\]
\noi Since ${ p<r<p^{*}}$, there exists $\rho>0$ small enough and $\eta>0$ such that $\mc E_{\la}(u)> \eta$ with $\|u\|=\rho$ for all $u\in D^{2,p}(\mb R^N)$.
\noi Define \[c= \displaystyle\inf_{\gamma\in \Gamma} \max_{t\in [0,1]} \mc E_{\la}(\ga t),\] where $\Gamma(t)=\{\gamma: \gamma\in(C[0,1], D^{2,p}(\mb R^N)): \ga(0)=0, \ga(1)= u_1\}.$
\noi Then $c>0$. Lemma \ref{L6} yields that $\mc E_{\la}$ satisfies the assumption of the mountain pass Lemma, see \cite[Theorem 2.1]{b1}. Thus there exists $u_2\in D^{2,p}(\mb R^N)$ such that $\mc E_{\la}(u_2)=c>0$ and $\mc E^{\prime}_{\la}(u_2)=0$. Hence, $u_2$ is a non-trivial solution of $(\mc P_{\la})$ different from $u_1$.\qed\\

\noi {\bf {To prove Theorem \ref{th4}:}} We will use Kranoselskii's genus theory \cite{59Krasnoselskii1968TopologicalMI}.
\noi Let $X$ be a real Banach space and $\sum$ the family of the set $E \subset X\setminus \{0\}$ such that $E$ is closed in $X$ and symmetric with respect to $0$, i.e.
	\[\sum = \{E\subset X\setminus \{0\}:  E\;\mbox{is closed in}\; X\; \mbox{and} \;  E=-E\}.\]
	
\noi For each $E\in \sum$, we say genus of $E$ is a number $k$ denoted by $\gamma(E) =k$ if there is an odd map $h\in C(E,\mathbb R^N\setminus\{0\})$ and $k$ is the smallest integer with this property.

\begin{lemma}\label{3.2}(\cite{8})
Let $X=\mathbb R^N$ and $\partial \Omega$ be the boundary of an open, symmetric, and bounded subset $\Omega \subset \mathbb R^N$ with $0\in \Omega$. Then $\gamma(\partial \Omega)= N$.
\end{lemma}
\noi It follows from Lemma \ref{3.2}, $\gamma(\mathbb S^{N-1})=N$, where $\mb S^{N-1}$ the surface of the unit sphere in $\mathbb R^N$.

\noi Now, we will use the following Theorem to obtain the existence of infinitely many solutions of $(\mc P_{\la})$.
 \begin{theorem} \label{3.1}(\cite{8}, \cite{10})
Let $I\in C^{1}(X,\mathbb R)$ be an even functional satisfying $(PS)$ condition. Furthermore
\begin{enumerate}
\item $I$ is bounded from below and even.
\item There is a compact set $E\in \sum $ such that $\gamma(E)=n$ and $\displaystyle\sup_{u\in \sum} I(u)< I(0)$, 
\end{enumerate}
then $I$ has at least $n$ pairs of distinct critical points and their corresponding critical values are less than $I(0)$.
 \end{theorem}

\noi {\bf {Proof of Theorem \ref{th4}:}} Let $\{e_1, e_2, \cdots\}$ be a Schauder basis of $D^{2,p}(\mathbb R^N)$. Foe each $n\in \mathbb N$, define $E_n= span\{e_1, e_2,\cdots, e_n\}$, the subspace of $D^{2,p}(\mathbb R^N)$ generated by $e_1, e_2,\cdot\cdot\cdot, e_n$. 
Define $L^{r}(\mb R^N, f)= \{u: \mb R^N\ra \mb R| \int_{\mb R^N} f(x)|u(x)|^r dx <\infty\}$, endowed with the norm
\[\|u\|_{r,f}= \left( \int_{\mb R^N} f(x)|u|^r dx\right)^{\frac {1}{r}}.\]

\noi By assumption $(f_1)$, one can easily see that $E_n$ can be continuously embedded  into $L^{r}(\mb R^N, f)$. As we know that all the norms are equivalent on a finite dimensional Banach space. Thus there exist a constant $C(n)$ depending on $n$ such that for all $u\in E_n,$
\[\left(\int_{\mathbb R^N} |\Delta u|^p dx\right)^{\frac 1p} \leq C(n) \left(\int_{\mb R^N} f(x)|u|^r dx\right)^{\frac 1r}. \]

\noi Then
\begin{align*}
\mc E_{\la}(u) &= \frac{b}{p\theta }\|u\|^{p\theta} - \frac{1}{2p_{\mu}^*}\int_{\mb R^N}\int_{\mb R^N} \frac{|u(y)|^{p_{\mu}^*}|u(x)|^{p_{\mu}^*} }{|x-y|^{\mu}} dxdy- \la \int_{\mb R^N} f(x)|u|^r dx\\
&\leq  \frac{b}{p\theta }\|u\|^{p\theta} -  \la C(n)\|u\|^{r}\\
&= \left( \frac{b}{p\theta }\|u\|^{p\theta-r} -  \la C(n)\right)\|u\|^r
\end{align*}

\noi Choose $R>0$ be a constant such that $\frac{b}{p\theta} R^{p\theta-r}<\la C(n).$
Hence for all $0<t<R$,
\[\mc E_{
\la}(u)\leq\left( \frac{b}{p\theta }t^{p\theta-r} -  \la C(n)\right) t^r\leq  \left( \frac{b}{p\theta }R^{p\theta-r} -  \la C(n)\right) R^r<0 = \mc E_{\la}(0),\]
for all $u\in K:= \{u\in E_n : \|u\|= t\}$. Then it follows that
\[\sup_{u\in K} \mc E_{\la}(u)<0= \mc E_{\la}(0).\]

\noi Clearly, $E_n$ and $\mb R^N$ are isomorphic and $K$ and $\mathbb S^{N-1}$ are homeomorphic. Therefore, we conclude that $\gamma(K)=n$ by Lemma \ref{3.2}. Moreover, $\mc E_{\la}$ is bounded below and satisfies $(PS)_c$ condition by Lemmas \ref{41} and \ref{L6}. Thus, Theorem \ref{3.1} give $\mc E_{\la}$ has at least $n$ pair of distinct critical points. The arbitrariness of $n$ yields that $\mc E_{\la}$ has infinitely many pairs of distinct solutions. 

\noi Let $u\in D^{2,p}(\mathbb R^N)\setminus \{0\}$ be a solution of $(\mc P_{\la})$. Then
\[b \|u\|^{p\theta}= \int_{\mb R^N}\int_{\mb R^N} \frac{|u(y)|^{p_{\mu}^*}|u(x)|^{p_{\mu}^*} }{|x-y|^{\mu}} dxdy+  \la \int_{\mb R^N} f(x)|u|^r dx. \]
\noi Using this together with \eqref{e3} and \eqref{n3} yield,
\[b \|u\|^{p\theta} \leq  S_{H,L}^{-\frac{2p_{\mu}^{*}}{p}} \|u\|^{2 p_{\mu}^{*}} + \la \|f\|_{\frac{p^*}{p^*-r}} S^{-\frac{r}{p}}\|u\|^r.\]
 
\noi Since $b> 2^p S_{H,L}^{-\frac{2p_{\mu}^{*}}{p}}$, $p\theta = 2 p_{\mu}^{*}$,  we obtain
\[(b- S_{H,L}^{-\frac{2p_{\mu}^{*}}{p}}) \|u\|^{p\theta} \leq \la \|f\|_{\frac{p^*}{p^*-r}} S^{-\frac{r}{p}}\|u\|^r.\]
\noi implies
\[\|u\|\leq \left[\frac{\la \|f\|_{\frac{p^*}{p^*-r}}}{S^{\frac{r}{p}}\left(b- S_{H,L}^{-\frac{2p_{\mu}^{*}}{p}}\right)}\right]^{\frac{1}{p\theta -r}}.\]
\noi Hence the proof is complete.\qed
\section{Case: \texorpdfstring{ $p\theta <2p_{\mu}^*$}{ }}
\noi This section is  divided into three subsections $5.1$, $5.2$ and $5.3$. In subsections $5.1$, $5.2$ and $5.3$, we give the proofs of our main Theorems for $1<r<p$, $r=p$ and $p\theta\leq r<p^*$ respectively.
\subsection{ \texorpdfstring{$1<r<p$}{ }}
In this subsection,  we first recall the definition of genus and then to prove Theorem 1.1, we use a result by Kajikiya see [\cite{kajikiya}, Theorem 1], which is an extension of the symmetric mountain pass theorem.
\begin{definition} 
Let $X$ be a Banach space, and $A$ be a subset of $X$. The set $A$ is said to be symmetric if $u \in A$ implies $-u \in A.$ For a closed symmetric set $A$ which does not contain the
origin, we define a genus $\gamma (A)$ of $A$ by the smallest integer $k$ such that there exists an
odd continuous mapping from $A$ to $\mathbb R^k \setminus \{0\}$. If there does not exist such  $k$, we define
	$\gamma (A)=\infty$ . Moreover, we set $\gamma (\emptyset)=0$ .
	\end{definition}
For any $n\in\mathbb N$, let us define the set $\Sigma_n$ as 
\[\Sigma_n:=\{A\;:\; A\subset X \text{ is closed symmetric }, 0\not\in A,\; \gamma (A)\geq n \}.\] 	
\begin{theorem}\label{sym mt}
Let $X$ be an infinite dimensional Banach space and $I\in C^1(X,\mb R)$. Suppose that the following hypotheses hold.
\begin{itemize}
\item[$(\mathcal A_1)$] The functional $I$ is even and bounded from below in $X$, $I(0)=0$ and $I$ satisfies the local Palais-Smale condition.
\item [$(\mc A_2)$] For each $n\in \mathbb N,$ there exists $A_n\in \Sigma_n$ such that $$\sup_{u\in A_n} I(u)<0.$$
\end{itemize} Then $I$ admits a sequence of critical points $\{u_n\}$ in $X$ such that $u_n\not =0$, $I(u_n)\leq 0$ for each $n$  and $u_n \to 0$ in $X$ as $n\to\infty$.
\end{theorem}

	\noi {\bf Proof of Theorem \ref{th1} :}
From the hypotheses, it follows that $ \mc E_{\al,\la}$ is even and $\mc {\mc E}_{\al,\la}(0)=0.$ Also Lemma \ref{cc} ensures that $ \mc E_{\al,\la}$ satisfies the  $(PS)_c$ condition for all $c<0$. But observe that, $\mc E_{\al,\la}$ is not bounded from below in $D^{2,p}(\mb R^N)$. So, for applying Theorem \ref{sym mt}, we use a truncation technique.\\
 Let $w\in D^{2,p}(\mb R^N)$. Then equations \eqref{e3} and \eqref{n3} yield
 \begin{align}\label{mpp}
\mc E_{\al,\la} (u)&=\frac{a}{p}\displaystyle\int_{{\mathbb R^N}} |\Delta u|^{p}dx+\frac{ b}{p\theta}\left( \int_{\mathbb R^N}|\Delta u|^pdx\right)^{p\theta } -\frac{\la}{r}\int_{\mathbb R^N}f(x)|u(x)|^r dx \notag\\
	&\quad\quad-\frac {\al}{2p_{\mu}^*}\int_{{\mathbb R^N}}\left(\int_{{\mathbb R^N}} \frac{|u(y)|^{p_{\mu}^*}}{|x-y|^\mu}dy\right)|u(x)|^{p_{\mu}^*}dx\notag\\
	&\geq \frac ap \|u\|^p-\frac \la r S^{-\frac{r}{p}} \|f\|_{\frac{p^*}{p^*-p}}\|u\|^{r}-\frac{\al}{2p_{\mu}^*} S_{H,L}^{-\frac{2p_{\mu}^*}{p}}\|u\|^{2p_{\mu}^*}\notag\\
	&:= C_1 \|u\|^p-\la C_2\|u\|^{r}-C_3\al \|u\|^{{2p_{\mu}^*}}.
\end{align}  Define the function $\mc K: \mb [0,\infty)\to\mb R$ as 
\begin{align*}
\mc K(t)=C_1 t^p-\la C_2t^{r}-C_3\al t^{{2p_{\mu}^*}}.
\end{align*} 
Since $1<r<p$, for any $\al>0$ we can choose $\la_0$  sufficiently small such that for all $\la\in (0,\la_0)$ there exist $0<t_1<t_2$ so that $\mc K<0$ in $(0,t_1)$, $\mc K>0$ in $(t_1,t_2)$ and $\mc K<0$ in $(t_2,\infty)$.

\noi Similarly, for any $\la>0$ we can choose $\al_0$  sufficiently small such that for all $\al\in (0,\al_0)$ there exist $0<t_{1}^{\prime}<t_{2}^{\prime}$ so that $\mc K<0$ in $(0,t_{1}^{\prime})$, $\mc K>0$ in $(t_{1}^{\prime},t_{2}^{\prime})$ and $\mc K<0$ in $(t_{2}^{\prime},\infty)$.
\noi Therefore $\mc K(t_1)=0=\mc K(t_2)$ and $\mc K(t_{1}^{\prime})=0=\mc K(t_{2}^{\prime})$. 
Next, we choose a non-increasing function $\mc I\in C^\infty([0,\infty),[0,1])$ such that
\begin{equation*}
	\mc I(t)=
	\begin{cases}
		&1  \mbox{\;\; if}\ t\in [0,t_1] \\
		&0  \mbox{\;\; if}\ t\in [t_2,\infty).
	\end{cases}
\end{equation*}
and set $\Phi(u):=\mc I(\|u\|).$  Now we define the truncated functional $\hat{\mc E}_{\al,\la}: D^{2,p}(\mb R^N)\to \mb R$ of $\mc E_{\al,\la}$ as
\begin{align}\label{ii}
	\hat{\mc E}_{\al,\la}(u)&:= \frac{a}{p}\displaystyle\int_{{\mathbb R^N}} |\Delta u|^{p}dx+\frac{ b}{p\theta}\left( \int_{\mathbb R^N}|\Delta u|^pdx\right)^{\theta}\\ &\qquad\quad-\Phi (u)\frac{\la}{r}\int_{\mathbb R^N}f(x)|u|^r dx -\Phi(u)\frac {\al}{2p_{\mu}^*}\int_{{\mathbb R^N}}\left(\int_{{\mathbb R^N}} \frac{|u|^{p_{\mu}^*}}{|x-y|^\mu}dy\right)|u|^{p_{\mu}^*} dx.\notag
\end{align} 
Then, it can easily seen that $\hat{\mc E}_{\al,\la}$ satisfies the following:
\begin{enumerate}
\item  $\hat{\mc E}_{\al,\la}\in C^1(D^{2,p}(\mb R^N),\mb R) $, $\hat{\mc E}_{\al,\la}( 0)=0$. 
\item  $\hat{\mc E}_{\al,\la}$ is even, coercive and bounded from below in $D^{2,p}(\mb R^N)$.
\item Let $c<0$, then for any $\la>0$ there exists $\overline{\Lambda}>0$ such that for all $\al\in (0,\overline{\Lambda})$, $\hat{\mc E}_{\al,\la}$ satisfies the Palais-Smale condition, by Lemma \ref{l2}.
\item Let $c<0$, then for any $\al>0$ there exists $\underline{\Lambda}>0$ such that $\hat{\mc E}_{\al,\la}$ satisfies the Palais-Smale condition for all \(\al\in (0, \underline{\Lambda})\), by Lemma \ref{l2}.

\item If $\hat{\mc E}_{\al,\la}( u)<0,$ then $\|u\|\leq t_1$ and $\hat{\mc E}_{\al,\la}(u)=\mc E_{\al,\la}(u)$.
\end{enumerate}
  For any $n\in\mb N$, we consider $n$ numbers of disjoint open sets denoted by $V_j$, $j=1,2,\cdots n$ with $\cup_{j=1}^n V_j\subset \Om$, where $\Om\not=\emptyset$ is given as in Theorem \ref{th1}. Now we choose $u_j\in D^{2,p}(\mb R^N)\cap C_0^\infty(V_j)\setminus\{0\}$, with $\|u_j\|=1$ for each $j=1,2,\cdots,n$. Set $$X_n=span\{u_1,u_2,\cdots,u_n\}.$$  Now we claim that there exists $0<\varrho_n<t_1,$ sufficiently small such that \begin{align}\label{claim1}
 	m_n:=\max\{\hat{\mc E}_{\al,\la}(u):u\in X_n,~\|u\|=\varrho_n\}< 0.
 \end{align}
  Suppose that \eqref{claim1} does not hold. Then there exists a sequence $\{u_k\}:=\{u_k^{(n)}\}$ in $X_n$ such that 
\begin{align}\label{f_11}\|u_k\|\to\infty;\; \hat{\mc E}_{\al,\la}(u_k)\geq 0.
\end{align} 
Let's set\[w_k=\frac{u_k}{\|u_k\|}.\] Then $w_k\in D^{2,p}(\mb R^N)$ and $\|w_k\|=1.$ Since $X_n$ is finite dimensional, there exists $w\in X_n\setminus\{0\}$ such that
\begin{align*}
	w_k&\to w \text{\;\; strongly with respect to \;\;} \|\cdot\|;\\
	w_k(x)&\to w(x) \text{\;\; a.e. in \;\;} \mb R^N.
\end{align*} As $w\not\equiv 0$, we get $|u_k(x)|\to\infty$ as $k\to\infty$. Thus, as $k\to\infty$,
\begin{align*}
&\frac{1}{\|u_k\|^{p\theta }}\int_{\mb R^N}\int_{\mb R^N} \frac{|u_k(x)|^{{p_{\mu}^*}}|u_k(y)|^{{p_{\mu}^*}}}{|x-y|^\mu}dxdy\\
&=\int_{\mb R^N}\int_{\mb R^N} \frac{|u_k(x)|^{{p_{\mu}^*}-\frac{p\theta}2}|u_k(y)|^{{p_{\mu}^*}-\frac{p\theta}2}}{|x-y|^\mu} |w_k(x)|^{\frac{p\theta}2}|w_k(y)|^{\frac{p\theta}2}dxdy\to \infty.
\end{align*}
Using this together with \eqref{ii}, we obtain
\begin{align*}
	\hat{\mc E}_{\al,\la}(u_k)&\leq \frac ap \|u_k\|^p+\frac{b}{p\theta }\|u_k\|^{p \theta}-\frac{\al}{2{p_{\mu}^*}}\int_{\mb R^N}\int_{\mb R^N} \frac{|u_k(x)|^{{p_{\mu}^*}}|u_k(y)|^{{p_{\mu}^*}}}{|x-y|^\mu}dxdy\\
	&\leq \|u_k\|^{p\theta }\left(\left(\frac ap +\frac{b}{p\theta}\right)-\frac{\al}{2{p_{\mu}^*}}\frac{1}{\|u_k\|^{p\theta }}\int_{\mb R^N}\int_{\mb R^N} \frac{|u_k(x)|^{{p_{\mu}^*}}|u_k(y)|^{{p_{\mu}^*}}}{|x-y|^\mu}dxdy\right)\\
	&\to-\infty
\end{align*} as $k\to\infty.$
This contradicts \eqref{f_11}. Thus, the claim is proved.

\noi Now choose $A_n:=\{u\in X_n\; :\; \|u\|=\varrho_n\}$. Clearly $\gamma (A_n)=n$ and $A_n$ is closed and symmetric, and hence $A_n\in\Sigma_n$ and also from \eqref{claim1}, $\sup_{u\in A_n}\hat{\mc E}_{\al,\la}(u)<0.$ Therefore, $\hat{\mc E}_{\al,\la} $ satisfies all the assumption in Theorem \ref{sym mt}. Thus, $\hat{\mc E}_{\al,\la}$ admits a sequence of critical points $\{u_n\}$ in $D^{2,p}(\mb R^N)$ such that $u_n\not =0$, $\hat{\mc E}_{\al,\la} (u_n)\leq 0$ for each $n\in \mb N$  and $\|u_n\| \to 0$ as $n\to\infty$. So, for $t_1>0$, there exists $ n_0\in\mathbb N$ such that for all $n\geq n_0$
it follows that $\|u\|<t_1$ which yields that $\hat{\mc E}_{\al,\la}(u_n)= {\mc E}_{\al,\la}(u_n)$ for all $n> n_0.$ This concludes the proof of the theorem.\qed
\subsection{ \texorpdfstring{$r=p$}{ } and  \texorpdfstring{ $\al=1$}{ }}
In this subsection, we use the following $\mb Z_2$-symmetric version of mountain
 pass theorem due to  \cite{r51rabinowitz1986minimax}, to prove Theorem \ref{th12}.
\begin{theorem}\label{sm}
		Let $X$ be an infinite dimensional Banach space with $X= Y\oplus Z$, where $Y$ is finite dimensional and let \(I \in C^{1}(X,\mathbb R)\) be an even functional with $I(0)=0$ such that the following conditions hold.
		\begin{enumerate}
			\item[$(\mc I_1)$] There exist positive constants \(\rho, \al>0\) such that $I(u)\geq\al$ for all \(u\in \partial B_{\rho}(0) \cap Z\);
			\item[$(\mc I_2)$] There exists $c^*> 0$ such that \(I\) satisfies the $(PS)_c$ condition for $0<c< c^*$;
			\item[$(\mc I_3)$] For any finite dimensional subspace $\tilde{X}\subset X$, there exists \(R= R(\tilde{X})>0\) such that $I(u)\leq 0$ for all \(u\in\tilde{X}\setminus B_{R}(0)\).
		\end{enumerate}
		Assume that $Y$ is $k-$dimensional and $Y= span\{v_1, v_2, \cdots, v_k\}$. For $n \geq k$, inductively choose \(v_{n+1}\not\in Y_n := span\{v_1, v_2, \cdots, v_n\}.\) Let \(R_n =R(Y_n)\) and \(D_n = B_{R_n}(0)\cap Y_k\). Define
		\[G_n :=\{h \in C(D_n, X): h \;\mbox{is odd and}\; h(u)=u,\; \forall\; u\in \partial B_{R_n}(0)\cap Y_n\}\]
		and
		\begin{align}\label{et}
		\Gamma_j := \{h(\overline{D_n\setminus E}) : h \in G_n, n\geq j, E\in \sum_{n-j},\;\mbox{and}\; \gamma(E)\leq n-j\},
\end{align}
		\[{\sum}_{n} := \{E: E\subset X \mbox{is closed symmetric}, 0\not\in E, \gamma(E)\geq n\},\]
		where $\ga(E)$ is Krasnoselskii's genus of $E$. For each $j\in \mathbb N$, set
		\[c_j :=\inf_{K\in \Gamma_j} \max_{u\in K} I(u).\]
		Then \(0<\alpha\leq c_j\leq c_{j+1}\) for $j>k$ and  if $j>k$ and $c_j< c^*$, then \(c_j\) is a critical value of $I$. Furthermore, if $c_j=c_{j+1}=\cdots=c_{j+m}=c< c^*$ for $j>k$, then $\gamma(K_c)\geq m+1$, where
		\[K_c= \{u\in X : I(u)=c\; \mbox{and} \;I^{\prime}(u)=0\}.\]
	\end{theorem}
Now we show that $\mc E_{\la} $ satisfies all the hypotheses of Theorem \ref{sm}, when $r=p$. 
\begin{lemma}\label{smg}
Let  $\al=1$ and $r=p $.	Then  $\mc E_{ \la}$ satisfies the conditions $(\mc I_1)$-$(\mc I_3)$ of Theorem \ref{sm} for all $\la\in (0,\,aS\|f\|^{-1}_{\frac{p^*}{p^*-p}}). $
\end{lemma}
\begin{proof} We first show the hypotheses $(\mc I_1)-( \mc I_{3})$ of Theorem \ref{sm}.
\\
{\it $(\mathcal I_1):$} For $u\in D^{2,p}(\mb R^N),$ arguing similarly as in \eqref{mpp}, we have
\begin{align*}
	\mc E_{ \la}( u)&\geq  \frac{\|u\|^p}{p}\left( a- \la S^{-1} \|f\|_{\frac{p^*}{p^*-p}}\right)-\frac{1}{2p_{\mu}^*}  S_{H,L}^{-\frac{2p_{\mu}^*}{p}}\|u\|^{2p_{\mu}^*}.
\end{align*} 
Now for $\la<aS\|f\|^{-1}_{\frac{p^*}{p^*-p}}$,  we can choose $\|u\|=l<<1$ so that $\mc E_{ \la}(u)\geq \mc K>0.$\\
{\it $(\mc I_2):$} It follows from Lemma \ref{ce}.\\
{\it $(\mc I_3):$} To show this, first claim that for any finite dimensional subspace
$Y$ of $D^{2,p}(\mb R^N)$ there exists $ R_0= R_0(Y)$ such that $ \mc E_{\la}( u)<0$ for all $u\in D^{2,p}(\mb R^N)\setminus B_{ R_0} (Y),$ where $B_{R_0}(Y)=\{u \in D^{2,p}(\mb R^N): \|u\|\leq  R_0\}.$ Fix $\phi\in D^{2,p}(\mb R^N),\;\|\phi\|=1.$ For $t>1$, we get
\begin{align}\label{sm1}
	\mc E_{\la}( t\phi)&\leq \frac ap t^{p} \|\phi\|^p+  \frac{t^{p\theta} b}{p\theta}\|\phi\|^{p\theta }
	-\frac{1}{{2p_{\mu}^*}} t^{2p_{\mu}^*}\int_{\mb R^N}\int_{\mb R^N} \frac{|\phi(x)|^{{p_{\mu}^*}}|\phi(y)|^{{p_{\mu}^*}}}{|x-y|^\mu}dxdy\notag\\
	&\leq C_4  t^{p\theta }\|\phi\|^{p\theta }
	-C_5  t^{2p_{\mu}^*}\|\phi\|_{\mu}^{2p_{\mu}^*}
\end{align}
Since $Y$ is finite dimensional, all norms are equivalent on $Y$, which yields that there exists some constant $C(Y)>0$ such that $C(Y)\|\phi\|\leq\|\phi\|_{\mu}.$ Therefore from \eqref{sm1}, we obtain

\begin{align*}
	\mc E_{ \la}( t\phi)&\leq C_4  t^{p\theta}
	-C_5 (C(Y))^{2p_{\mu}^*}t^{2p_{\mu}^*}\|\phi\|^{2p_{\mu}^*}\\
	&= C_4  t^{p\theta}
	-C_5 (C(Y))^{2p_{\mu}^*} t^{2p_{\mu}^*}\to-\infty
\end{align*}
as $t\to\infty$.
Hence, there exists $ R_0>0$ large enough such that $\mc E_{\la}( u)<0$ for all
$u\in D^{2,p}(\mb R^N)$ with $\|u\|={R}$ and $ R\geq R_0$. Therefore, $\mc E_{\la}$ satisfies the assertion $(\mc I_3)$.
\end{proof}

\begin{lemma}\label{ss}
There exists a non-decreasing sequence $\{S_n\}$ of positive real numbers, independent of $\la$ such that for any $\la>0$, we have
\[c_n^\la:=\inf_{A\in \Gamma_n}\max_{u\in A} \mc E_{ \la}(u)<S_n,\] where $\Gamma_n$ is defined in \eqref{et}.
\end{lemma}

\begin{proof}
\noi Recalling the definition of $c_n^\la$ and \eqref{ii}, we get 
 \[ c_n^\la\leq \inf_{A\in \Gamma_n}\max_{u\in A}\left[\frac ap \|u\|^p+\frac{b}{p\theta}\|u\|^{p\theta}-\frac{1}{2{p_{\mu}^*}}\|u\|_{\mu}^{2p_{\mu}^*}\right]:=S_n\] 
 Then clearly from the definition of $\Gamma_n$, it follows that $S_n<\infty$ and $S_n\leq S_{n+1}$.
\end{proof}

\noi {\bf Proof of Theorem \ref{th12}:} From the hypotheses of the theorem it follows that $\mc E_{\al,\la}$	is even and we have $\mc E_{\al, \la}(0)=0.$  
From the Lemma \ref{ss}, we can choose, $\hat a>0$ sufficiently large such that for any $a>\hat a$,
\[\sup_n S_n<\left(\frac{1}{ p}- \frac{1}{2p^*_\mu}\right)\left(a  S_{H,L}\right)^{\frac{2N-\mu}{N-\mu +2p}}:= c^*,\] that is,
\[c_n^\la<S_n<\left(\frac{1}{ p}- \frac{1}{2p^*_\mu}\right)\left(a  S_{H,L}\right)^{\frac{2N-\mu}{N-\mu +2p}}.\] Hence, for all $\la\in(0,\,aS\|f\|^{-1}_{\frac{p^*}{p^*-p}})$ and $a>\hat a$, we have
\[0<c_1^\la\leq c_2^\la\leq\cdots\leq c_n^\la<S_n<c^*.\]
Now by Theorem \ref{sm}, we infer that the levels $c_1^\la\leq c_2^\la\leq\cdots\leq c_n^\la$ are critical values of $\mc I_\la.$ Therefore, if $c_1^\la<c_2^\la<\cdots< c_n^\la$, then $\mc E_{ \la}$ has at least $n$ number of critical points. Furthermore, if $c_j^\la=c_{j+1}^\la$ for some $j=1,2,\cdots,k-1$, then again Theorem \ref{sm} yields that $A_{c_j^\la}$ is an infinite set. Hence, the problem $(\mc P_{\la})$ has infinitely many solutions. 
Consequently, the problem $(\mc P_{\la})$ has at least $n$ pairs of solutions in $D^{2,p}(\mb R^N)$.\qed

\subsection{\texorpdfstring{$p\theta \leq r<p^*$}{ } and \texorpdfstring{$\al=1$}{ }}
\noi In this subsection, we prove Theorem \ref{th13} using Theorem \ref{sm}. For that, first we show $\mc E_{\la}$ verifies all the hypotheses of Theorem \ref{sm}, when $p\theta \leq r<p^*$.
\begin{lemma}\label{ft}
Let $\al=1$ and $p\theta \leq r<p^*$.	Then  $\mc E_{\la}$ satisfies the conditions $(\mc I_1)$-$(\mc I_3)$ of Theorem \ref{sm} in the following cases:\begin{enumerate}
    \item If $r=p\theta$ and $\lambda\in \left(0,bS^{\theta}||f||_{\frac{p^*}{p^*-p\theta}}^{-1}\right)$.
    \item If $p\theta < r<p^*$ and $\la>0$.
\end{enumerate} 
\end{lemma}
\begin{proof}
 Let $u\in D^{2,p}(\mb R^N)$ with $\|u\|<1$.  Using the similar arguments as in \eqref{mpp},  we get\begin{align}\label{s1}
					\mc E_{\la}( u)
					&\geq \frac ap \|u\|^p+\frac{b}{p\theta}||u||^{p\theta}-\frac \la r S^{-\frac{r}{p}} \|f\|_{\frac{p^*}{p^*-r}}\|u\|^{r}-\frac{1}{2p_{\mu}^*}  S_{H,L}^{-\frac{2p_{\mu}^*}{p}}\|u\|^{2p_{\mu}^*}.
				\end{align}
If $r=p\theta$
\begin{align}\label{s2}
                    \mc E_{\la}( u)
					&\geq \frac ap \|u\|^p+\frac{1}{p\theta}\left(b- \la S^{-\theta} \|f\|_{\frac{p^*}{p^*-p\theta}}\right)||u||^{p\theta}-\frac{1}{2p_{\mu}^*}  S_{H,L}^{-\frac{2p_{\mu}^*}{p}}\|u\|^{2p_{\mu}^*}\notag\\
                    &\geq \frac ap \|u\|^p-\frac{1}{2p_{\mu}^*}  S_{H,L}^{-\frac{2p_{\mu}^*}{p}}\|u\|^{2p_{\mu}^*}.
                \end{align}
\noi   In the view of  \eqref{s1}, \eqref{s2}, $ p<r$ and $p<p_{\mu}^*$, we can choose $0<\rho<1$  sufficiently small so that, we  obtain  for all $u \in D^{2,p}(\mb R^N)$ with $\|u\|=\rho$, $\mc E_{\la}( u) \geq \beta>0$ for some $\beta>0$  depending on $\rho$. Thus $(\mc I_1)$ holds.\\
$(\mc I_2)$ follows from Lemma \ref{cc}, since $c^{**}>0$ and  for $(\mc I_3)$, the argument follows similarly to that of Lemma \ref{smg}.
\end{proof}

\noi {\bf Proof of Theorem \ref{th13}} Using Lemma \ref{ft} and proceeding in a way similar to Lemma \ref{ss} and in Theorem \ref{sm}, we can conclude that the problem $(\mc P_{\la})$ has at least $n$ pairs of solutions for all $\la>0$.\qed\\
\noindent{\bf Acknowledgements :} 
The first and second authors would like to thank the Science and Engineering Research Board, Department of Science and Technology, Government of India, for the financial support under the grant  SRG/2022/001946 and SPG/2022/002068, respectively. 
\bibliographystyle{plain}
\bibliography{references.bib}

\end{document}